\pdfoutput=1
\documentclass[11pt]{amsart}
\usepackage{amsmath}
\usepackage[all]{xy}
\usepackage{amssymb}
\usepackage{mathrsfs}
\usepackage{cite}
\usepackage{hyperref}
\usepackage{amsfonts}
\usepackage{todonotes}
\theoremstyle{plain}
\newtheorem{thm}{Theorem}[section]
\newtheorem{prop}[thm]{Proposition}
\newtheorem{lem}[thm]{Lemma}
\newtheorem{cor}[thm]{Corollary}

\theoremstyle{remark}
\newtheorem{rem}[thm]{Remark}
\newtheorem*{acks}{Acknowledgements}

\theoremstyle{definition}
\newtheorem{defn}[thm]{Definition}

\newtheorem{eg}[thm]{Example}

\theoremstyle{conjecture}
\newtheorem{conj}[thm]{Conjecture}
\numberwithin{equation}{section}
\def\1A{\mathcal{A}}
\def\B{\mathcal{B}}
\def\C{\mathsf{C}}
\def\D{\mathsf{D}}
\def\E{\mathsf{E}}
\def\F{\mathsf{F}}
\def\G{\mathsf{G}}
\def\H{\mathsf{H}}
\def\I{\mathsf{I}}

\def\K{\mathsf{K}}
\def\L{\mathsf{L}}
\def\M{\mathsf{M}}
\def\N{\mathsf{N}}

\def\R{\mathsf{R}}

\def\T{\mathcal{T}}
\def\U{\mathsf{U}}
\def\V{\mathsf{V}}
\def\W{\mathsf{W}}
\def\X{\mathsf{X}}
\def\Y{\mathsf{Y}}
\def\Z{\mathsf{Z}}
\def\a{\mathsf{a}}
\def\b{\mathsf{b}}
\def\c{\mathsf{c}}
\def\d{\mathsf{d}}
\def\e{\mathsf{e}}
\def\f{\mathsf{f}}

\def\i{\mathsf{i}}
\def\j{\mathsf{j}}
\def\k{\mathsf{k}}
\def\l{\mathsf{l}}
\def\m{\mathsf{m}}
\def\n{\mathsf{n}}
\def\o{\mathsf{o}}
\def\p{\mathsf{p}}
\def\q{\mathsf{q}}

\def\s{\mathsf{s}}
\def\t{\mathsf{t}}

\def\x{\mathsf{x}}
\def\y{\mathsf{y}}

\def\A1{\mathsf{A}}
\def\2B{\mathsf{B}}
\def\op{\mathsf{op}}
\def\Hom{\mathsf{Hom}}

\def\Ext{\mathsf{Ext}}
\def\Tor{\mathsf{Tor}}
\def\Mor{\mathsf{Mor}}

\def\dg{\mathsf{dg}}
\def\HH{\mathsf{HH}}
\def\Per{\mathsf{Per}}

\def\Perf{\mathsf{Perf}}

\def\DbX{\mathsf{D}^{\mathsf{b}}(\mathsf{X})}
\def\DbY{\mathsf{D}^{\mathsf{b}}(\mathsf{Y})}

\def\qch{\mathsf{qch}}
\def\Qch{\mathsf{Qch}}
\def\cof{\mathsf{cof}}
\def\Cone{\mathsf{Cone}}
\def\mod{\mathsf{mod}}
\def\th{\mathsf{th}}

\def\PBs{\mathsf{PBs}}

\title{some remarks of Hochschild homology and semi-orthogonal decompositions}
\begin{document}

\author{Xun Lin}
\address{Yau Mathematical Sciences Center, Tsinghua university, Beijing China}
\email{lin-x18@mails.tsinghua.edu.cn}
\begin{abstract}
   Given a nontrivial semi-orthogonal decomposition $\Perf(\X)=\langle \mathcal{A},\mathcal{B}\rangle$, and assume that the base locus of $\omega_{\X}$ is a proper closed subset, it was proved by Kotaro Kawatani and Shinnosuke Okawa that all skyscraper sheaves $\k(\x)$ with $\x\notin \mathsf{Bs}\vert \omega_{\X}\vert$ belong to exactly one and only one of the components. It is natural to ask which one it is, and whether we can determine this by certain linear invariants. In this note we use Hochschild homology of derived category of coherent sheaves with support to provide another proof that if the $-\n^{\th}$ Hochschild homology of a component is nonzero, then the skyscraper sheaves we consider above belong to such component, which was originally proved by Dmitrii Pirozhkov \cite[Lemma 5.3]{pirozhkov2020admissible}. Furthermore, we prove a conjecture proposed by Kuznetsov
   about classifying $\n$-Calabi-Yau admissible subcategory of $\Perf(\X)$ ($\dim \X=\n$) for certain projective smooth variety $\X$  if we put more assumptions to the Calabi-Yau categories. Finally we remark that the additive invariants of derived category with support could provide more linear obstructions to semi-orthogonal decompositions.
\end{abstract}

\maketitle

\tableofcontents
\section{Introduction}
Derived category of coherent sheaves play a central role in homological algebra, Mirror symmetry, and representation theory. One way to study the
derived category is decomposing itself as a triangulated category. The natural decomposition is orthogonal decomposition. However, it is too restrictive since if the variety is connected, the derived category of coherent sheaves has no nontrivial orthogonal decompositions. The notion of semi-orthogonal decomposition is more flexible and it turns out that there are many examples, and much geometric information are encoded \cite{alex2014semiorthogonal}.

\par

Hochschild homology theory of admissible subcategories of derived category of a projective smooth variety and $\dg$ categories were studied by many people. For admissible subcategories, A.\ Kuznetsov constructed an intrinsic Hochschild homology theory in paper \cite{K}. One of the fantastic properties is that the Hochschild homology is additive for semi-orthogonal decompositions. This provided some interesting obstruction for semi-orthogonal decompositions. For example, if $\Perf(\X)$ admits a full exceptional collection, then $\X$ must be of Hodge-Tate. $\dg$ categories are considered as noncommutative counterpart of varieties, they are more flexible than triangulated categories. For instance, many homological theories can be defined at the framework of $\dg$ categories. Furthermore, some classical conjectures in algebraic geometry can be generalized to certain $\dg$ categories, see \cite{tabuada2019noncommutative}. The Hochschild homology (or additive invariants) of $\dg$ categories were
extensively studied by B.\ Keller (or G.\ Tabuada), see \cite{KELLER1998223}, \cite{KELLER19991} (or \cite{tabuada2004une}).  Kuznetsov's additivity of semi-orthogonal decompositions can be obtained in framework of $\dg$ categories, even though it is not intrinsic --- we use the $\dg$ enhancement which would be non-canonical. However, it is expected that Hochschild homology of $\dg$ categories should provide more flexible possibilities to the study of semi-orthogonal decompositions or derived categories itself.

 \par

 Kotaro Kawatani and Shinnosuke Okawa proved that if there is a semi-orthogonal decomposition $\Perf(\X):=\langle \1A,\B\rangle$, then either support of all objects in $\1A$ are in $\Z:=\2B\s\vert\omega_{\X}\vert$ (then $\k(\x)\in \B$ with $\x\in \U:=\X\setminus \Z$) or support of all objects in $\B$ are in $\Z$ (then $\k(\x)\in \1A$ with $\x\in \U:=\X\setminus \Z$) \cite{article3}. A natural question is that under the assumption that $\Z$ being proper closed subset, could we determine to which component the skyscraper sheaves belong by certain linear invariants ?

 \par

In this note, we apply new techniques to study the semi-orthogonal decomposition problems via Hochschild homology of $\dg$ categories. The key observation is that
 $\Perf_{\Z}(X)$ is a unique enhanced triangulated category \cite{Antieau2018OnTU}. We provide an interesting answer to the question raised above by Hochschild homology of derived categories of coherent sheaves with support. We write $\n$ as the dimension of $\X$. Note that the theorem bellow was originally proved in \cite[Lemma 5.3]{pirozhkov2020admissible}.

 \begin{thm}(= Theorem \ref{A})\label{A1}
  \cite[Lemma 5.3]{pirozhkov2020admissible}Let $\X$ be a projective smooth variety of dimension $\n$. Suppose there is a nontrivial semi-orthogonal decomposition $\Perf(\X)=\langle\mathcal{A},\mathcal{B}\rangle$ with $\HH_{-\n}(\mathcal{B})\neq 0$. Then the support of any object in $\mathcal{A}$ is contained in $\Z=\2B\s\vert\omega_{\X}\vert$. The skyscraper sheaves $\k(\x)$ with $\x\in \X\setminus \Z$ belongs to $\mathcal{B}$. Furthermore, $\HH_{-\n}(\mathcal{A})= 0$, $\HH_{-\n}(\Perf(\X))\cong \HH_{-\n}(\mathcal{B})$. If $\mathcal{B}$ is a $\n$ Calabi-Yau category, then it is indecomposable.
\end{thm}

\begin{rem}
  Note that the statement still holds when changing the role of $\mathcal{A}$ and $\mathcal{B}$. $\mathcal{A}$ should be regarded as the smaller piece in the components of the semi-orthogonal decomposition.
\end{rem}

 Surprisingly, it is closed related to a conjecture proposed by Kuznetsov that any $\n$-Calabi-Yau category of $\Perf(\X)$ ($\dim\X=\n$) is equivalent to a derived category of variety $\Y$, and $\X$ is blow-up of $\Y$ \cite{Kuznetsov_2016}. The essence of the conjecture is that the behaviour of semi-orthogonal decompositions that contain $\n$-Calabi-Yau category as a component should be the same with that of
 the examples constructed from blow-ups. However, Kuznetsov's conjecture is far from being true, see ``Compact $Hyper-K\ddot{a}hler$ Categories'' \cite[Section 5]{rol2015compact}. It was shown in the paper that there are infinite many geometric $4$-folds containing $4$-Calabi--Yau (connected) categories which are not derived category of a projective smooth variety. Hence, we should not expect to recover geometric information by less categorical information. However, we can achieve this if there are more assumptions about the Calabi-Yau subcategory for certain varieties.

\begin{thm}(= Theorem \ref{Main1})\label{Main}
  Let $\Y$ be a smooth projective Calabi-Yau variety of dimension $\n$, $\f:\X\rightarrow \Y$ is the bolw-up of $\Y$ over points $\{\y_{\i}\}$. Define the distinguish objects
  $\D_{\i}=\L\f^{\ast}\k(\y_{\i})$. Let $\mathcal{B}$ be an $\n$ Calabi-Yau admissible subcategory of $\Perf(\X)$. If all $\D_{\i}\in \mathcal{B}$, then $\mathcal{B}=\L\f^{\ast}\Perf(\Y)$.
\end{thm}

\begin{rem}\ \\
(1) The assumption is something like marking points of the category $\mathcal{B}$.   \\
(2) Assume there is an element $\phi$ of $\mathsf{Aut}(\Perf(\X))$ which maps the distinguished objects $\D_{\i}=\L\f^{\ast}(\k(\x_{\i}))$ to the one whose support is disjoint to exceptional divisors (base locus of $\K_{\X}$). Then the assumption in the theorem always holds for $\mathcal{B}$, see Theorem \ref{D}.

\end{rem}

Even though the original Kuznetsov conjecture failed, we obtain some
interesting theorems which satisfy the original intuition. Thus, we expect that the Kuznetsov conjecture is still true for surfaces because the birational geometry is clearer for surfaces.

\begin{cor}\ \\
\begin{enumerate}
  \item (= Theorem \ref{B})\label{B1}) Let $\X$ be a projective smooth variety of dimension $\n$. Suppose $\dim \H^{0}(\X,\omega_{\X}) \geq 2$, then any $\n$-Calabi--Yau admissible subcategory of $\Perf(\X)$ is not a derived category of a smooth projective variety.\\
  \item (= Theorem \ref{C}\label{C1}) Let $\X$ be a projective smooth variety. Suppose its derived category $\Perf(\X)$ admits a
 nontrivial semi-orthogonal decomposition with component $\mathcal{B}$ being  $\n$-$\C\Y$ category. Here $\n$ is dimension of $\X$. Then, the base locus $\Z$ of the canonical bundle is a proper closed subset. Furthermore, any skyscraper sheaves $\k(\x)$ with $\x\in \X\setminus \Z$ belongs to $\mathcal{B}$.
 \end{enumerate}
\end{cor}
   One of interesting problem is to find weak Calabi--Yau categories which are non-admissible subcategory of $\DbY$, where $\Y$ is a projective smooth variety. It turns out that there are series of examples.
 \begin{thm}(= Theorem \ref{CY})\label{CY1}
  Let $\X$ be a projective smooth variety of dimension $\n$. $\Z$ is a closed subscheme of $\X$.
  \begin{enumerate}
    \item  If $\X$ is a Calabi-Yau variety, then $\Perf_{\Z}(\X)$ is a weak Calabi-Yau category, and it is not an admissible subcategory of derived category of projective smooth varieties.
    \item If $\Z$ consist of points, then $\Perf_{\Z}(\X)$ is a weak Calabi-Yau category, and it is not an admissible subcategory of derived category of projective smooth varieties.
  \end{enumerate}
\end{thm}

\begin{rem}
We must note that probably the easiest way to prove these facts is to prove that $\Perf_{\Z}(\X)$ is not smooth, then by Orlov' result \cite{Orlov2016SMOOTHAP}, it can not be an admissible subcategory of $\Perf(\Y)$ for some smooth projective $\Y$.
\end{rem}
 As for proof of Theorem \ref{A1}, we use recent theorem proved by Kotaro Kawatani and Shinnosuke Okawa (see Proposition \ref{sodsupport}) that if there is a semi-orthogonal decomposition $\Perf(\X)=\langle \1A,\B\rangle$, then either support of all objects in $\1A$ are in $\Z:=\2B\s\vert \omega_{\X}\vert$ or support of all objects in $\B$ are in $\Z$. Then we can localize the semi-orthogonal decomposition
of $\Perf(\X)$ to that of $\Perf_{\Z}(\X)$. Applying Hochschild homology theory of $\dg$ categories, we prove Theorem \ref{A1}. We essentially use the fact that $\HH_{-\n}(\Perf_{\Z}(\X))=0$, $\dim \X=\n$ for any proper closed subset $Z$ Corollary \ref{-nHoch}. Finally, using the fact that Hochschild homology of admissible $\n$-Calabi-Yau categories have non-vanishing $-\n$ Hochschild homology \cite[Corollary 5.4]{Kuznetsov_2019}, we have Theorem~\ref{CY1}.

\par

\subsection*{Notations}\label{note}
In this note, we assume $\X$ to be smooth projective over the base field $\k =\mathbb{C}$. Since $\DbX\cong \Perf(\X)$ for projective smooth variety $\X$, we won't distinguish $\DbX$ and $\Perf(\X)$. We say a $\k$-linear triangulated category to be weak Calabi-Yau if it admits Serre functor of shifting (the triangulated structure). It is not a classical definition of Calabi--Yau category.
\begin{acks}
The author would like to thank his supervisor Will Donovan for helpful disscusions. The author also thanks Baosen Wu, Dingxing Zhang, Jie Zhou, and Minghao Wang for helpful conversations. Partial results were obtained when the author attended The geometry of derived categories seminar, Liverpool University, UK. The author is grateful to Shinnosuke Okawa's talk at A.\ Bondal 60-th Birthday conference, 2021. It makes the author realize that Theorem \ref{A1} was proved by Dmitrii Pirozhkov \cite[Lemma 5.3]{pirozhkov2020admissible}.
\end{acks}

  \vspace{5mm}

\section{Hochschild Homology}
 In this section, we briefly recall the Hochschild homology for algebras, $\dg$ categories and unique enhanced triangulated categories.
As for notions of $\dg$ category, there is a great survey by B.\ Keller, on differential categories \cite{keller2006differential}. The Hochschild homology theory of $\dg$ algebras and $\dg$ categories is mainly from papers of B.\ Keller \cite{KELLER1998223} and \cite{KELLER19991}. The author introduces a Hochschild homology theory of unique enhanced triangulated categories, even though it may be well known for experts.
 \subsection{Hochschild homology of algebra}
  Given a $\k$ algebra $\A1$, we define the Hochschild complex of $\A1$ to be
                   $$\C(\A1):=\A1\otimes^{\L}_{\A1^{\e}}\A1 $$
and define Hochschild homology of $\A1$ to be $\HH_{\bullet}(\A1)=\H_{\bullet}(\C(\A1))=\H_{\bullet}(\A1\bigotimes^{\L}_{\A1^{\e}}\A1).$
Here $\A1^{\e}:=\A1\otimes \A1^{\op}$.
We can calculate derived tensor by bar resolution of $\A1$ as module $\A1^{\e}$ $\colon$ The degree $\n$ term of bar complex is $\A1\otimes \A1^{\bigotimes \n}\otimes \A1$, differential of the complex is defined as $$\d_{\n}(\a_{0}\otimes \a_{1}\otimes\cdots\otimes\a_{\n}\otimes \a_{\n+1} )=\sum_{\i=0}^{\n}(-1)^{\i}\a_{0}\otimes \a_{1}\otimes\cdots \otimes(\a_{\i}\cdot \a_{\i+1})\otimes \cdots \otimes\a_{\n+1}.$$

\begin{rem}
  Using bar resolution, the Hochschild complex $\C(\A1)$ is as follows: the degree $\n$ part is $\A1^{\otimes\n+1}$. The differential $$\d_{\n}(\a_{0}\otimes \a_{2}\otimes \cdots \otimes \a_{\n})=\oplus_{\i=0}^{\n}(-1)^{\i}\a_{0}\otimes \cdots\otimes\a_{\i}\a_{\i+1}\cdots \otimes \a_{\n}+(-1)^{\n}\a_{\n}\a_{0}\otimes \cdots \otimes \a_{\n-1}$$
\end{rem}

\begin{eg}\label{eg 1.1}
Consider the polynomial algebra $\A1=\k[\x_{1},\x_{2},\cdots,\x_{\n}]$. Then $$\HH_{\bullet}(\A1)=\oplus_{\i=0}^{\n}\Omega^{\i}(\A1).$$
\end{eg}

 For a differential graded algebra over $\k$, we still define the Hochschild complex in the same way, but it is a little subtle. The bar resolution is no longer a projective resolution as a differential graded algebra, but the total complex of bar resolution turns out to be a proper one. Thus for differential graded algebra $\Lambda$ over $\k$
 $$\HH_{\bullet}(\Lambda)\colon=\H_{\bullet}(\mathsf{Tot}(\mathsf{Bar}(\Lambda) \otimes \Lambda)).$$
 See B.\ Keller's paper \cite[section 1]{KELLER1998223} for a more explicit description.

 \subsection{Hochschild homology of differential graded category}
  \text Hochschild homology can be generalized to small (differential graded) categories. We sketch basic notions of the differential graded categories, the reader can refer to the survey, \cite[On differential graded categories]{keller2006differential}.

\begin{defn}
 The $\k$-linear category $\1A$ is called a $\dg$ category if
 $\Mor(\E,\F)$ is differential $\mathbb{Z}$-graded $\k$-vector spaces for every object $\E$, $\F$, $\G$ $\in$ $\1A$, and the compositions
 $$\Mor(\F,\E)\times \Mor(\G,\F)\rightarrow \Mor(\G,\E) $$
 are morphism of complexes and associative.
 Furthermore, there is a unit $\k\rightarrow \Mor(\E,\E)$. Note
that $\Mor(\E,\E)$ is a differential graded algebra because of the composition law.
\end{defn}

\begin{eg}\label{eg 1.3}
One of the basic example is $\C_{\dg}(\k)$, objects are complexes of $\k$ vector space, the morphism spaces
are redefined as follows$\colon$\\
Let $\E$, $\F$ $\in$ $\C_{\dg}(\k)$, define degree $\n$ piece to be $\Mor(\E,\F)(\n)\colon=\Pi \Hom(\E_{\i},\F_{\i+\n})$.
The $\n^{\
th}$ differential is given by $\d_{\n}(\f)= \d_{\E}\circ \f - (-1)^{\n}\f\circ \d_{\F}$, $\f\in \Mor(\E,\F)(\n)$.
\end{eg}

\begin{defn}
 We call $\F\colon \C\rightarrow \D$ a dg functor between dg category if it preserves degree of morphisms. We say $\F$ is quasi-equivalent if $\F$ induces isomorphisms on homologies of morphisms and equivalences of their homotopic categories.
\end{defn}

Consider the category of small $\dg$ categories, it is well known that there is a model structure with weak equivalence being quasi-equivalence\cite{tabuada2004une}\cite{article}. We denote
$\operatorname{Ho(dg-cat)}$ to be the localization of the category of small $\dg$ categories to quasi-equivalence. Recall that morphism of $\mathcal{A}$ to $\B$ in $\operatorname{Ho(dg-cat)}$ can be represented by $\mathcal{A}\leftarrow \mathcal{A}_{\cof}\rightarrow \mathcal{B}$.

Now back to the definition of Hochschild homology, the reader can refer to Bernhard Keller's paper \cite{KELLER19991}. Let $\mathcal{A}$ be a small $\dg$ category.
Consider the total complex $$\bigoplus_{\E_{0},\E_{1},\E_{2},\dots,\E_{\n}}  \quad\mathcal{A}(\E_{0},\E_{1})\otimes \mathcal{A}(\E_{1},\E_{2})\otimes\dots\otimes \mathcal{A}(\E_{\n},\E_{0})$$ where $\E_{0},\E_{1},\E_{2},\dots,\E_{\n}$ run over the set of  objects of $\mathcal{A}$. The tensor product is in the graded sense.
The vertical degeneracy differential is defined as follows:\\
$$\d_{\i}(\f_{0},\f_{1},\dots,\f_{\i-1},\f_{\i}\dots \f_{\n})=
\begin{cases}
 (\f_{0},\cdots,\f_{\i}\f_{\i-1},\cdots,\f_{\n})& \text{if} \quad \i < \n\\
(-1)^{\n+\d}(\f_{\n}\f_{0},\cdots,\f_{\n-1})& \text{if} \quad \i=\n
\end{cases}$$

The cyclic operator is given by $$\t_{\n}(\f_{\n-1},\dots,\f_{0})=(-1)^{\n+\d}(\f_{0},\dots,\f_{\n-1}).$$
$\d=\operatorname{deg}(\f_{0})\cdot(\sum_{\i=1}^{\n-1}\operatorname{deg} (\f_{\i}))$. See \cite[1.3]{KELLER19991}. The Hochschild homology of $\mathcal{A}$ is defined to be homology of the total complex.

\par

\begin{rem}\label{rem2.6}
Following paper \cite{KELLER19991}, we can associate a mixed complex whose underline complex is defined above for a small $\dg$ category. A mixed complex can be regarded as complex with degree $-1$ operator $\mathsf{B}$ such that $\mathsf{B}^{2}=0$, $\d\mathsf{B}+\mathsf{B}\d=0$. It is equivalent to modules of a $\dg$ algebra $\Lambda$, $\Lambda:= \k[\varepsilon]/(\varepsilon^{2})$, degree $\varepsilon$ is $-1$ and $\d\varepsilon =0$.
Hence, when we talk about Hochschild complexes, we always refer to the mixed complexes in mixed derived category (the derived category of modules over a $\dg$ algebra $\Lambda$). The Hochschild homology is the homology of the mixed complex for the original differential $\d$.
\end{rem}

 \begin{rem}\ \\
 1. A differential graded algebra can be regarded as a differential graded category with one object and the two definitions of Hochschild homology given above coincide.\\
 2. Some people define the Hochschild homology of a mixed complex to be the opposite degree homology of the mixed complex, which is compatible with the Hochschild homology of an algebra. In this paper, we define the Hochschild homology of a mixed complex as the homology of the complex and we regard the Hochschild homology of an algebra concentrates at non-positive degrees.
\end{rem}

\begin{defn}
 The Hochschild homology can be defined for a localizing pair $\langle \mathcal{B}_{0},\mathcal{B}\rangle$ of differential graded categories. It is defined as $\Cone(\C(\mathcal{B}_{0})\rightarrow \C(\mathcal{B}))$ in the mixed derived category. The cone is taken in the mixed derived category, see Remark \ref{rem2.6}. Here the $\dg$ categories are exact (or strong pre-triangulated) $\dg$ categories and $\mathcal{B}_{0}$ is the exact dg subcategory, see\cite[2.1]{KELLER19991}. For simplicity, the exact dg category means the shift $[\n]$ and the cone exist in degree $0$ levels. It is stronger than being a pre-triangulated category since the latter one requires the existence of shifts and cones in the homotopic level. We say the Verdier quotient $[\mathcal{B}]/[\mathcal{B}_{0}]$ the associated triangulated category of the localizing pair $\langle \mathcal{B}_{0},\mathcal{B}\rangle$.
\end{defn}

\subsection{Hochschild homology of unique enhanced triangulated categories}
\begin{defn}
Let $\mathcal{T}$ be a $\dg$ category. Consider Yoneda embedding, $\Y\colon \mathcal{T}\hookrightarrow \dg-\mod(\mathcal{T})$. Let $\f\colon \X\rightarrow \Z$ be of degree zero and closed morphism, $\X, \Z \in \mathcal{T}$. It would happen that the object $\Y^{\A1}[\m]$($\A1\in \mathcal{T}$) and $\Cone(\Y^{\X}\rightarrow \Y^{\Z})$ do not come from the image of $\Y$, i.e they are not represented. If they are represented in the homotopic category, we call $\mathcal{T}$ pre-triangulated $\dg$ category. If they are represented in $\dg-\mod(\mathcal{T})$ and $\mathcal{T}$ has zero object, we call $\mathcal{T}$ strong pre-triangulated (or exact $\dg$ category). The typical example of a strong pre-triangulated category is $\C_{\dg}(\mathcal{A})$, the dg category of complexes of $\mathcal{A}$ where $\mathcal{A}$ is an abelian category.
\end{defn}

\begin{lem}\label{exactdglift}
Let $\mathcal{A}$ be a $\dg$ category. There is a universal exact $\dg$ envelop $\Ext(\mathcal{A})$ of $\mathcal{A}$. That is , there is an embedding $\mathcal{A}\hookrightarrow \Ext(\mathcal{A})$, and for any exact $\dg$ category $\mathcal{B}$ and dg functor, $\mathcal{A}\rightarrow \mathcal{B}$, there is a unique extended $\dg$ functor from $\Ext(\mathcal{A})$ to $\mathcal{B}$.
\end{lem}

\begin{proof}
Consider Yoneda embedding $\Y: \mathcal{A}\rightarrow \dg-\mod(\mathcal{A})$. Let $\Ext(\mathcal{A})$ be the full subcategory of $\dg-\mod(\mathcal{A})$ with objects being generated by shifting and cones of $\Y([\mathcal{A}])$ in the homotopic category. $\Ext(\mathcal{A})$ satisfies the universal property. Clearly, we have equivalent $\Y\colon [\mathcal{A}]\rightarrow [\Ext(\mathcal{A})]$ if $\mathcal{A}$ is pre-triangulated.
\end{proof}

\begin{defn}
 The triangulated category $\mathcal{T}$ has a unique enhancement if it has pre-triangulated enhancement. Moreover, for any two enhancements $\mathcal{T}_{1}$ and $\mathcal{T}_{2}$, there exists a quasi-functor or
  a morphism in $\operatorname{Ho(dg-cat)}$ (according to Bertrand To\"en \cite{article}, they are equivalent) which induces an equivalence of homotopic categories .
\end{defn}

\begin{lem}\label{enhancedperf}
  Let $\X$ be a projective smooth variety and $\Z$ is a closed scheme of $\X$. The triangulated categories $\Perf(\X)$ and $\Perf_{\Z}(\X)$ are unique enhanced triangulated categories \cite{2010JAMS...23..853L}\cite[Corollary 2]{Antieau2018OnTU}. Here, $\Perf_{\Z}(\X)$ is the full subcategory of $\Perf(\X)$ whose objects support in $\Z$.
\end{lem}

\begin{proof}
  The statement is true for more general $\X$, which is also well known for experts. For the case of projective smooth variety $\X$, there is an easier way to see this. Since $\Perf(\X)\cong \D^{\b}(\mathsf{coh}(\X))$, and $\Perf_{\Z}(\X)\cong \D^{\b}(\mathsf{coh}_{\Z}(\X))$, according to \cite[Corollary 2]{Antieau2018OnTU}, they are unique enhanced triangulated categories.
\end{proof}

\begin{defn}
 Let $\T$ be a unique enhanced triangulated category. We define Hochschild homology of $\T$ as Hochschild homology of its pre-triangulated $\dg$ enhancement. Note that two different pre-triangulated $\dg$ enhancement are isomorphic in $\operatorname{Ho(dg-cat)}$, which implies there is a chain of quasi-equivalences of $\dg$ categories connected these two enhancement. Thus, the Hochschild homology of $\T$ is independent of choice of pre-triangulated enhancement because Hochschild homology is a derived Morita invariant.
\end{defn}

 We calculate the Hochschild homology of $\Perf(\X)$ and $\Perf_{\Z}(\X)$ via this definition. Before doing these, we need some lemmas.

 \begin{lem}\label{Homologydgpair}
 Let $\langle\mathcal{B}_{0},\mathcal{B}\rangle$ be a localizing pair, and denote its Hochschild complex by $\C_{\bullet}$. Suppose $\mathcal{T}:= [\mathcal{B}/\mathcal{B}_{0}]\cong [\mathcal{B}]/[\mathcal{B}_{0}]$ is a unique enhanced triangulated categories, then $\C(\mathcal{T})\cong \C_{\bullet}$ in mixed derived category.
\end{lem}

 \begin{proof}

Since $\mathcal{B}$ is a pre-triangulated category, the Drinfeld dg quotient $\mathcal{B}/\mathcal{B}_{0}$(see Drinfeld \cite{D}) is a pre-triangulated dg category \cite[Lemma 1.5]{2010JAMS...23..853L}.
Consider the exact $\dg$ envelope $\mathcal{D}$, we have a diagram
$$\mathcal{B}_{0}\rightarrow \mathcal{B}\rightarrow \mathcal{B}/\mathcal{B}_{0}\hookrightarrow \mathcal{D}$$
which induces an exact sequence of triangulated categories:
$$[\mathcal{B}_{0}]\rightarrow [\mathcal{B}]\rightarrow [\mathcal{B}/\mathcal{B}_{0}]\cong [\mathcal{D}].$$
According to Bernhard Keller\cite[Thm 4(c) ]{KELLER19991}, there is a triangle in the mixed derived category,
$$\C(\mathcal{B}_{0})\rightarrow \C(\mathcal{B})\rightarrow \C(\mathcal{D})\rightarrow \C(\mathcal{B}_{0}[1]).$$
Then, by definition, we have an isomorphism of Hochschild complexes $C_{\bullet}\cong C(\mathcal{T})$ in the mixed derived category.
 \end{proof}

\begin{prop}\label{dgpairsX}
   (\cite[Section 5.7]{1998math......2007K}) Let $\X$ be a smooth projective variety, and $\Z$ is a closed subset of $\X$. Consider the natural exact $\dg$ pairs $\Per(\X)$: They are the category of perfect complexes and its full subcategory of acyclic perfect complexes whose associated triangulated categories is $\Perf(\X)$. Similar notation $\Per_{\Z}(\X)$ for the natural $\dg$ pair of support case. The Hochschild complexes of these $\dg$ pairs are isomorphic to $\R\Gamma(\X,\M(\mathcal{O}_{\X}))$ and $\R\Gamma_{\Z}(\X,\M(\mathcal{O}_{\X}))$ respectively, where $\M(\mathcal{O}_{\X})$ is the sheaf of bar complex.
\end{prop}

\begin{thm}\ \label{geometryformula}
Let $\X$ be a projective smooth variety, and $\Z$ is a closed subset. Then

   $$\HH_{\i}(\Perf_{\Z}(\X))\cong \bigoplus_{\p-\q=\i}\H_{\Z}^{\p}(\X,\Omega_{\X}^{\q}).$$

\end{thm}

\begin{proof}
  According to Lemma \ref{Homologydgpair} and Proposition \ref{dgpairsX}, it suffices to compute $\R\Gamma_{\Z}(\X,\M(\mathcal{O}_{\X}))$. If $\X$ is smooth, we have a morphism of complexes which is a $\k$-quasi-isomorphism \\
 \centerline{\xymatrix{\cdots\ar[r]&\mathcal{O}_{\X}\otimes \mathcal{O}_{\X}\otimes \mathcal{O}_{\X}\ar[r]^{\mathsf{bar}}\ar[d]^{\d}&\mathcal{O}_{\X}\otimes \mathcal{O}_{\X}\ar[r]^{\mathsf{bar}}\ar[d]^{\d}&\mathcal{O}_{\X}\ar[r]^{0}\ar[d]^{\d}&0\ar[d]^{\d}\ar[r]&0\\
\cdots\ar[r]&\Omega^{2}_{\X}\ar[r]^{0}&\Omega^{1}_{\X}\ar[r]^{0}&\Omega^{0}_{\X}\ar[r]^{0}&0\ar[r]&0
}}
Here $\M(\mathcal{O}_{\X})$ is $\k$-quasi-isomorphic to sheaf version of the bar complex, $\d(\f_{1}\otimes \f_{2}\otimes \cdots \otimes\f_{\n}):=\frac{1}{\n!}\f_{1}(\d\f_{2}\bigwedge \cdots \bigwedge\d\f_{\n})$.
Note that it is firstly defined via presheaves, and then induces morphism of sheaves.
Locally it is a $\k$-quasi-isomorphism by usual HKR theorem for regular algebras. Hence the differential induces a $\k$-quasi-isomorphism of complexes. Then formula follow easily.
\end{proof}

\begin{cor}\label{-nHoch}
 Let $\X$ be a projective smooth variety of dimension $\n$, and $\Z$ is a proper closed subscheme of $\X$. Then, $\HH_{-\n}(\Perf_{\Z}(\X))=0$.
\end{cor}

\begin{proof}
 According Theorem \ref{geometryformula},
 $$\HH_{-\n}(\Perf_{\Z}(\X))\cong \oplus_{\p-\q=-\n}\H_{\Z}^{\p}(\X,\Omega^{\q}_{\X})= \H^{0}_{\Z}(\X,\Omega^{\n}_{\X})=0.$$

\end{proof}
\vspace{5mm}

\section{Computation of Hochschild homology via Fourier-Mukai kernels}
In this section, we will first introduce basic notions of  Fourier-Mukai transform in algebraic geometry, the standard reference is \cite[D.\ Huybrechtz's book]{book2}. Then we recompute the geometric formulas of Hochschild homology in Theorem \ref{geometryformula} by using techniques of Fourier-Mukai kenels .

\subsection{Fourier-Mukai Functors}
\begin{defn}
 Let $\X$, $\Y$ be projective smooth varieties, $\E \in \Perf(\X\times\Y)$, the induced functor $$\F\colon \Perf(\X) \rightarrow \Perf(\Y)$$
 given by $\F(\bullet)= \R\p_{2\ast}(\p_{1}^{\ast}\bullet \otimes^{\L}\E)$ is called Fourier-Mukai functor associated to the kernel $\E$.
 \begin{center}
 $$\xymatrix@C50pt{&\X\times \Y\ar[dl]_{\p_{1}}\ar[dr]^{\p_{2}}&\\
            \X&& \Y}$$
 \end{center}

\end{defn}

\begin{eg}
 $\I\d : \X \rightarrow \X$ is given by kernel $\Delta_{\ast}\mathcal{O}_{\X}$.
 This follows from the projection formula
$$ \R\p_{2\ast}(\p_{1}^{\ast}\E\otimes \Delta_{\ast}\mathcal{O}_{\X})
   =\R\p_{2\ast}(\Delta_{\ast}(\Delta^{\ast}\p_{1}^{\ast}\E\otimes \mathcal{O}_{\X}))
   =\E.$$
\begin{center}
 $$\xymatrix@C50pt@R15pt{&\X\ar[d]^{\Delta}\ar[ddl]_{\i\d}\ar[ddr]^{\i\d}&\\
 &\X\times \X\ar[dl]^{\p_{1}}\ar[dr]_{\p_{2}}&\\
 \X&&\X
 }$$
 \end{center}

The readers can check that derived functors $\R\f_{\ast}\colon\Perf{\X} \rightarrow \Perf(\Y)$ and $\L\f^{\ast}\colon \Perf(\Y)\rightarrow \Perf(\X)$ which are induced by morphism $\f\colon \X\rightarrow \Y$ are both Fourier-Mukai functors with some kernels, including Serre functor and the functor tensor with line bundles.
\end{eg}

\begin{eg}
 Let $\F_{1}\in \Perf(\X\times\Y)$ and $\F_{2}\in \Perf(\Y\times \Z)$ be the kernels respect to Furier- Mukai functors $\f_{1}$ and $\f_{2}$. Then the composition $\f_{2}\circ \f_{1}$ is a Fourier-Mukai functor with kernel $$\pi_{\X\Z\ast}(\pi_{\X\Y}^{\ast}(\F_{1})\otimes \pi_{\Y\Z}^{\ast}(\F_{2})).$$
\end{eg}

 It is natural to ask whether the exact functors between derived categories are of Fourier-Mukai functor, it is true for fully faithful functor.

\begin{prop}\label{faithfulfuncor}(D\ .Orlov \cite[Thm 3.4]{orlov1996equivalences})\\
 Let $\F \colon \Perf(\X) \hookrightarrow \Perf(\Y)$ be an exact fully faithful functor which has adjoint (left or right adjoint). Then it is isomorphic to a Fourier-Mukai functor. In particular, the corresponding Kernel is unique up to isomorphism.
 \end{prop}

\begin{rem}
 The assumption that the fully faithful functor obtain adjoint functors can be skipped due to deep results of Bondal, Van den Bergh \cite{bondal2002generators}.
 \end{rem}

 \begin{eg}
  The identity functor $\i\d\colon \Perf(\X)\rightarrow \Perf(\X)$ is Fourier-Mukai functor with kernel $\mathcal{O}_{\Delta}$. The kernel corresponding to $\i\d$ is unique up to isomorphism.
 \end{eg}
  The theory of which kinds of functors are Fourier-Mukai functor was generalized to derived categories with support \cite[Theorem 1.1]{canonaco_stellari_2014}. Assume $\Z$ have no components of zero dimension. Clearly the identity functor of $\Perf_{\Z}(\X)$ satisfies assumption in \cite[Theorem 1.1]{canonaco_stellari_2014}. We prove a theorem which will be used to compute Hochschild homology of $\Perf_{\Z}(\X)$.

 \begin{thm}\label{kernelsupport}
  Let $\X$ be a smooth projective variety, and $\Z$ is a closed subscheme with no zero dimensional components. The object $\Delta_{\ast}\R\underline{\Gamma_{\Z}}(\mathcal{O}_{\X})$ is a kernel of $\i\d\colon \Perf_{\Z}(\X)\rightarrow \Perf_{\Z}(\X)$. If $\E \in \D_{\Z\times\Z,\qch}(\X\times\X)$ such that $\Phi_{\E}\cong \i\d$, then  $\E\cong\Delta_{\ast}\R\underline{\Gamma_{\Z}}(\mathcal{O}_{\X})$.
 \end{thm}

  Before proving Theorem \ref{kernelsupport}, we propose some interesting lemmas about projection formula and support of objects, which will be used in later calculations too.

 \begin{lem}\label{projformula}\cite[\href{https://stacks.math.columbia.edu/tag/01E6}{Tag 01E6}]{stacks-project}
  Let $\f: \X\rightarrow \Y$ be morphism of ringed space. Suppose
  f maps X homeomorphic into closed subset, then the general
  projection formula holds: $\E\in \mathsf{D}(\mathcal{O}_{\X})$, $\F\in \mathsf{D}(\mathcal{O}_{\Y})$. Then
  $$\R\f_{\ast}(\E\otimes^{\L} \L\f^{\ast}\F)\cong \R\f_{\ast}\E\otimes^{\L} \F.$$
  In particular, it is true for closed immersion.
\end{lem}

 \begin{rem}\
 \begin{enumerate}
   \item Derived Tensor product, pull back and pull forward maps can be defined in generality of $\mathcal{O}_{\X}$ module, namely using $\K$-injective and $\K$-flat resolution to define. See "resolution of unbounded complex"\cite{CM_1988__65_2_121_0}
   \item We will use the projection formula in case of closed immersion, open immersion and projection. More generally, the projection formula holds for quasi-separated quasi-compact morphism of schemes\cite[\href{https://stacks.math.columbia.edu/tag/08ET}{Tag 08ET}]{stacks-project} for $\E,\F \in \D_{\qch}(\bullet)$.
 \end{enumerate}

\end{rem}

\begin{lem}\label{support}
  Let $\X$ be an algebraic variety, $\F\in \mathsf{D}_{\mathsf{ch}}^{\mathsf{b}}(\X)$. Take a closed point $\x$, consider the embedding $\l_{\x}\colon \x\hookrightarrow \X$, then $\x$ is in the support of $\F$ if and only if $\mathbb{L}\l^{\ast}_{\x}\F\neq 0$.
  \end{lem}

\begin{proof}
  Suppose $\x$ is in the support of $\F$, but $\mathbb{L}\l^{*}_{\x}(\F)= 0$. Let $\x\in$ support of $\mathcal{H}^{\m}(\F)$ with maximal integer $\m$.
  Then, there is a canonical truncation:
  $\F_{\leq \m-1}\rightarrow \F\rightarrow \F_{> \m-1}\rightarrow \F_{\leq \m-1}[1] .$
  Since $\mathbb{L}\l_{\x}^{\ast}\F=0$, we have isomorphism
  $\mathbb{L}\l_{\x}^{\ast}\F_{>\m-1}\cong \mathbb{L}\l_{\x}^{\ast}\F_{\leq \m-1}[1]$.
  Since $\l_{\x}^{-1}\F_{>\m-1}\cong \l_{\x}^{-1}\mathcal{H}^{\m}(\F)[-\m]$, therefore $\mathbb{L}\l_{\x}^{\ast}\F_{>\m-1}\cong \l_{\x}^{-1}\mathcal{H}^{\m}(\F)[-\m]\otimes^{\L} \k(\x)\cong \mathbb{L}\l_{\x}^{\ast}\mathcal{H}^{\m}(\F)[-\m].$
  Take the flat resolution$\colon$
  $$\cdots\rightarrow \I_{1}\rightarrow \I_{0}\rightarrow \mathcal{H}^{\m}(\F).$$
  Apply the functor (without derived) $\l^{\ast}_{\x}$:
  $$\cdots\rightarrow \l^{\ast}_{\x}\I_{1}\rightarrow \l^{\ast}_{\x}\I_{0}\rightarrow 0.$$
  Then the $0^{\th}$ homology gives $\mathcal{H}^{\m}(\F)_{\x}\otimes \k(\x)$.
  Since by assumption, $\x\in$ support $\mathcal{H}^{\m}(\F)$, then $\mathcal{H}^{\m}(\F)_{\x}\neq 0$.
  Hence by Nakayama Lemma, $\mathcal{H}^{\m}(\F)_{\x}\otimes \k(\x)\neq 0$. Therefore, $\mathbb{L}\l_{\x}^{\ast}\F_{>\m-1}\cong \mathbb{L}\l_{\x}^{\ast}\F_{\leq \m-1}[1]\neq 0$ has non zero degree $\m$ homology. However, $\F_{\leq \m-1}[1]$ survives in degree less than $\m$, hence after taking a flat resolution, it survives in degree less than $\m$. But it means that $\mathbb{L}\l_{\x}^{\ast}(\F_{\leq \m-1}[1])$ has zero $\m$ homology, a contradiction.\\
   Suppose $\mathbb{L}\l^{\ast}_{\x}\F\neq 0$, but $\x$ $\notin$ support of $\F$. Since $\mathbb{L}\l_{\x}^{\ast}\F= \l_{\x}^{-1}\F\otimes_{\mathcal{O}_{\X,\x}}^{\L}\k(\x)$,
  and $\l_{\x}^{-1}\F$ is an acyclic complex, we have
  $ \mathbb{L}\l_{\x}^{\ast}\F\cong 0$, a contradiction.
  \end{proof}

\begin{rem}
 There is another proof using spectral sequence. Consider the spectral sequence
 $$\E_{2}^{\p,\q}= \Tor_{-\p}(\mathcal{H}^{\q}(\l_{\x}^{-1}\F),\k(\x)) \quad \Rightarrow \quad \mathcal{H}^{-\p-\q}(\mathbb{L}\l_{\x}^{\ast}\F).$$
Suppose $\mathbb{L}\l_{\x}^{\ast}\F\neq 0$, but $\x$ $\notin$ support $\F$. Then $\E_{2}^{\p,\q}=0$, a contradiction. Suppose $\x$ $\in$ support $\F$, but $\mathbb{L}\l_{\x}^{\ast}\F=0$. Let $\m$ be the maximal integer such that $\x$ $\in$ support of $\mathcal{H}^{\m}(\F)$. Then $\E_{2}^{0,\m} \neq 0$ and degenerates to $\E_{\infty}$, a contradiction.
\end{rem}

\begin{proof}[Proof of Theorem \ref{kernelsupport}]
  Let $\F \in \D_{\Z}(\Qch(\X))$. Actually, by \cite[Theorem 1.1]{canonaco_stellari_2014}), we need to check perfect objects $\F$, but here, it is true for more general $\F$. Since the diagonal morphism is a closed embedding, the projection formula
always holds for $\mathcal{O}_{\X}$ module.
   $$ \R\p_{2\ast}(\R\p_{1}^{\ast}\F\otimes^{\L} \Delta_{\ast}\R\underline{\Gamma_{\Z}}\mathcal{O}_{\X})
      \cong \R\p_{2\ast}\Delta_{\ast}((\mathbb{L}\Delta^{\ast}\R\p_{1}^{\ast}\F)\otimes^{\L} \R\underline{\Gamma_{\Z}}\mathcal{O}_{\X})
       \cong \F\otimes^{\L} \R\underline{\Gamma_{\Z}}\mathcal{O}_{\X}.$$

Denote $\j: \X\setminus \Z \rightarrow \X$, there is a triangle
         $$ \R\underline{\Gamma_{\Z}}\mathcal{O}_{\X}\rightarrow \mathcal{O}_{\X}\rightarrow \R\j_{\ast}\j^{\ast}\mathcal{O}_{\X}\rightarrow  \R\underline{\Gamma_{\Z}}\mathcal{O}_{\X}[1].$$
Tensor (derived sense) with $\F$, there is a new triangle
 $$ \R\underline{\Gamma_{\Z}}\mathcal{O}_{\X}\otimes^{\L} \F \rightarrow \mathcal{O}_{\X}\otimes^{\L} \F \rightarrow \R\j_{\ast}\j^{\ast}\mathcal{O}_{\X}\otimes^{\L} \F\rightarrow \R\underline{\Gamma_{\Z}}\mathcal{O}_{\X}\otimes^{\L} \F[1].$$
Clearly, $\R\j_{\ast}j^{\ast}\mathcal{O}_{\X}\otimes^{\L} \F \cong \R\j_{\ast}j^{\ast}\F\cong 0$, and hence $\R\underline{\Gamma_{\Z}}\mathcal{O}_{\X}\otimes^{\L} \F\cong \F.$ Suppose $\E\in \D(\Qch_{\Z\times\Z}(\X\times\X))$ such that $\Phi_{\E}\cong \i\d$. Then according to \cite[Lemma 5.3]{canonaco_stellari_2014}, $\E\in \D^{\b}(\Qch_{\Z\times\Z}(\X\times\X))$. Thus, according to \cite[Theorem 1.1]{canonaco_stellari_2014}, $\E\cong \Delta_{\ast}\R\underline{\Gamma_{\Z}}\mathcal{O}_{\X}$.
\end{proof}

\subsection{The computations}
\begin{thm}\label{gfkernel}
 Let $\X$ be a projective smooth variety and $\Z$ is a closed subscheme with no zero dimensional components. Then

  $$\HH_{\i}(\Perf_{\Z}(\X))\cong \Hom^{\i}(\mathcal{O}_{\X\times \X},\Delta_{\ast}\R\underline{\Gamma_{\Z}}(\mathcal{O}_{\X})\otimes^{\L}\Delta_{\ast}\R\underline{\Gamma_{\Z}}(\mathcal{O}_{\X}))\cong\bigoplus_{\p-\q=\i}\H_{\Z}^{\p}(\X,\Omega_{\X}^{\q}).$$

\end{thm}

\begin{rem}
  We assume $\Z$ to be closed subscheme with no zero dimensional components since the kernel for $\i\d$ is proved to be unique up to isomorphism in this case. We hope that it is true more generally.
\end{rem}

 Before proving Theorem \ref{gfkernel}, we need some preparations.

 \begin{lem}\label{dgalge}
   Let $\X$ be a smooth projective variety and $\Z$ is a closed subscheme of $\X$. Then the categories
   $\D_{\qch}(\X)$ and $\D_{\Z,\qch}(\X)$ have a compact generator respectively \cite[Theorem 6.8]{rouquier_2008}. In particular, they are derived equivalent to derived category of some $\dg$ algebras.
 \end{lem}

 \begin{proof}
   For derived category without support, it is well known. The case of support is similar. Let $\E$ be a compact generator of $\D_{\Z,\qch}(\X)$. We write $\E$ again after resolution to a $\K$-injective perfect complex. Define $\Lambda=\Hom_{\dg}(\E,\E)$. Then
   $$ \L\colon \mathsf{D}(\Lambda)\rightarrow \D_{\Z,\qch}(\X),\quad \quad \M \mapsto \E\otimes^{\L}_{\Lambda}\M.$$
 $$\R\Hom(\E,\cdot)\colon \D_{\Z,\qch}(\X)\rightarrow \mathsf{D}(\Lambda),\quad \quad \F \mapsto \R\Hom(\E,\F).$$
 define equivalence $\D(\Lambda)\cong\D_{\Z,\qch}(\X)$.
 \end{proof}

 \begin{lem}
   Let $\E$ be a compact generator of $\D_{\Z,\qch}(\X)$. Then $\E^{\vee}$ is also a compact generator too. In particular, $\E^{\vee}\boxtimes \E$ is a compact generator of $\D_{\Z\times\Z,\qch}(\X\times\X)$.
 \end{lem}

 \begin{proof}
 Again, this fact is well know for non support case. We provide a proof here which mimic the proof in  A.\ Bondal, Van den Bergh's paper \cite[Lemma 3.4.1]{bondal2002generators}.

 \begin{defn}
 Let $\mathcal{C}$ be a $\k$-linear triangulated category. We say $\E$ (could be set of objects) generates $\mathcal{C}$ if $\Hom(\E[\n],\c)=0$ for all integer $\n$ implies $\c=0$. $\E$ classical generates $\mathcal{C}$ if the minimal full triangulated subcategory of $\mathcal{C}$ containing $\E$ must be $\mathcal{C}$. The reader can refer to the paper \cite{bondal2002generators} for the definitions.
 \end{defn}

 \begin{lem}\label{compactgenerate}
 \cite{ASENS_1992_4_25_5_547_0} Assume $\mathcal{C}$ (triangulated $\k$-linear category) is compactly generated, that is $\mathcal{C}^{\c}$
 generates $\mathcal{C}$. Then a set of objects $\mathcal{E} \in \mathcal{C}^{\c}$ classically generates $\mathcal{C}^{\c}$ if and only if it generates $\mathcal{C}$.
 \end{lem}
 According to the Lemma \ref{compactgenerate}, it suffices to prove $\E^{\vee}$ classical generates
  $\Perf_{\Z}(\X)$.
  Consider $\F\in\Perf_{\Z}(\X)$, then $\F^{\vee}\in \Perf_{\Z}(\X)$, hence $\F^{\vee}\in$ $\langle \E \rangle$, taking dual, $\F\in$ $\langle \E^{\vee}\rangle$.

  \par

  Back to object $\E^{\vee}\boxtimes \E$, its support $\subseteq \Z\times \Z$,
  which can be easily proved by using Lemma \ref{support} above.
  Claim$\colon\E^{\vee}\boxtimes \E$ is a compact generator of $\D_{\Z\times\Z,\qch}(\X\times\X)$.
  It suffices to prove that $\E^{\vee}\boxtimes \E$ classical generates $\Perf_{\Z\times\Z}(\X\times\X)$ by Lemma \ref{compactgenerate}. But again by Neeman-Ravenel theorem, replace $\mathcal{C}$ by $\mathcal{C}^{\c}$, we conclude that $\E^{\vee}\boxtimes \E$ classical generates $\Perf_{\Z\times\Z}(\X\times\X)$ if and only if $\E^{\vee}\boxtimes \E$ generates $\Perf_{\Z\times\Z}(\X\times\X)$.\\
  Let $\W\in \Perf_{\Z\times\Z}(\X\times\X)$. Suppose $\Hom(\E^{\vee}\boxtimes \E[i], \W)=0$ for any integer $\i$.
  $$\Hom(\p_{1}^{\ast}\E^{\vee},\R\mathcal{H}\o\m_{\X\times \X}(\p^{\ast}_{2}\E, \W[-\i]))=0.$$
  For arbitrary integer $\m$, we obtain
  $$ \Hom(\E^{\vee}[\i+\m],\R\p_{1\ast}\R\mathcal{H}\o\m_{\X\times \X}(\p^{\ast}_{2}\E,\W[\m]))=0.$$
  $\R\p_{1\ast}\R\mathcal{H}\o\m_{\X\times \X}(\p^{\ast}_{2}\E,\W[\m])$ supports in $\Z$, the proof is as follows$\colon$Firstly, the support of $\R\mathcal{H}\o\m_{\X\times \X}(\p_{2}^{\ast}\E,\W[\m])$ is in $\Z\times \Z$, which follows from Lemma \ref{support}. We claim that $\R\p_{1\ast}(\F)$ supports in $\Z$ if $\F$ supports in $\Z\times \Z$. There is a commutative diagram, $\U=\X\setminus \Z$.\\
  \centerline{\xymatrix{U\times \X\ar[r]^{\j'}\ar[d]^{\p_{1\U}}&\X\times \X\ar[d]^{\p_{1}}\\
\U\ar[r]^{\j}&\X}}\\
According to flat base change theorem,
$\mathbb{L}\j^{\ast}\R\p_{1\ast}\F\cong \R\p_{1\U\ast}\mathbb{L}\j'^{\ast}\F\cong 0$, which means the support of $\R\p_{1\ast}\F$ is in $\Z$.
  Hence $\R\p_{1\ast}\R\mathcal{H}\o\m(\p^{\ast}_{2}\E, \W[\m])$ supports in $\Z$.

  \par

  Since $\E^{\vee}$ is a compact generator of $\D_{\Z,\qch}(\X)$,
  $$\R\p_{1\ast}\R\mathcal{H}\o\m_{\X\times \X}(\p^{\ast}_{2}\E,\W[\m+\n])=0.$$
  Take any affine open sub-scheme $\U$ of $\X$,
   we have $\Hom_{\U\times \X}(\p^{\ast}_{2}\E\mid_{\U\times \X},\W\mid_{\U\times \X}[\m+\n])\cong 0$, and then
  $\Hom_{\X}(\E,(\R\p_{2\ast}\W\mid_{\U\times \X}[\m])[\n])=0.$
 Since $\W\mid_{\U\times \X}$ supports in $(\Z\cap \U)\times \Z$ , $\R\p_{2\ast}\W\mid_{\U\times \X}[\m]$ supports in $\Z$ because of the same reason above. We obtain
  $$\R\p_{2\ast}(\W\mid_{\U\times \X}[\n])=0$$
  for any integer $\n$.
  Again taking any affine open sub-scheme of $\V$ of $\X$, we have
  $$\Gamma(\U\times \V,\W|_{\U\times \V})=0$$
  for any open affine $\U,\ \V$. Therefore, $\W\cong 0$.
 \end{proof}

 \begin{prop}
   Choose a compact generator $\E$ of $\D_{\Z,\qch}(\X)$. Let $\Lambda=\R\Hom(\E,\E)$. There is a commutative diagram\\
   \centerline {\xymatrix@R5pc@C6.5pc{\D_{\Z,\qch}(\X)\ar[r]^{Fourier- Mukai}\ar[d]^{\R\Hom(\E,\cdot)}&\D_{\Z,\qch}(\X)\ar[d]^{\R\Hom(\E,\cdot)}\\
\D(\Lambda)\ar[r]^{Furier-Mukai\quad \M}&\D(\Lambda)
}}
   The morphism in the upper row is Fourier-Mukai functor associated with $\M\otimes^{\L}_{\Lambda^{\op}\otimes \Lambda}\E^{\vee}\boxtimes \E$.
 \end{prop}

 \begin{proof}
 We need to prove:
$$\R\p_{2\ast}(\p^{\ast}_{1}\F\otimes^{\L} \E^{\vee}\boxtimes \E\otimes^{\L}_{\Lambda^{\op}\otimes \Lambda} \M)
\cong \R\Hom(\E,\F)\otimes^{\L}_{\Lambda^{\op}} \M\otimes^{\L}_{\Lambda}\E.$$
Firstly, we replace $\M$ with $\Lambda^{\op}\otimes \Lambda$, then by projection formula
$$\text{Left}=(\R\Hom(\E,\F)\otimes \E)\otimes^{\L}_{\Lambda^{\op}\otimes \Lambda}\Lambda^{\op}\otimes \Lambda=\text{Right}.$$
For general $\M$, we have semi-free resolution, without loss of generality, assume $\M$ semi-free.
That is, there is a filtration: $0\subset \phi_{1}\subset\phi_{2}\subset\cdots\phi_{\n}\subset\cdots\subset\M$, with quotient being direct sum (maybe infinite) of $\Lambda^{\op}\otimes \Lambda[\n]$. Then since the formula holds for $\Lambda^{\op}\otimes \Lambda[\n]$, and hypercohomology and tensor product commute with direct sum (infinite), it holds for $\phi_{\i}$. According to \cite[\href{https://stacks.math.columbia.edu/tag/09KL}{Tag 09KL}]{stacks-project}, there is a triangle
$$\oplus \phi_{\i}\rightarrow \oplus \phi_{\i}\rightarrow \M.$$
Thus, it holds for $\M$.
 \end{proof}

\begin{cor}\label{Nkernel}

 $\Lambda\otimes^{\L}_{\Lambda^{\op}\otimes \Lambda}\E^{\vee}\boxtimes \E$
  is the Fourier-Mukai kernel corresponding to identity functor of $\D_{\Z,\qch}(\X)$ and hence of $\Perf_{\Z}(\X)$. In particular, $\Lambda\otimes^{\L}_{\Lambda^{\op}\otimes \Lambda}\E^{\vee}\boxtimes \E\cong \Delta_{\ast}\R\underline{\Gamma_{\Z}}(\mathcal{O}_{\X})$ if $\Z$ has no zero dimensional components.

\end{cor}

\begin{proof}

  \begin{align*}
  \R\p_{2\ast}(\p_{1}^{\ast}\F\otimes^{\L} (\p_{1}^{\ast}\E^{\vee}\otimes^{\L}_{\mathcal{O}_{\X}} \p_{2}^{\ast}\E)\otimes^{\L}_{\Lambda^{\op}\otimes \Lambda}\Lambda)
& \cong  (\R\Hom(\E,\F)\otimes \E)\otimes^{\L}_{\Lambda^{\op}\otimes \Lambda}\Lambda \\
& \cong  \R\Hom(\E,\F)\otimes^{\L}_{\Lambda} \E \\
& \cong  \F.
  \end{align*}
By Theorem \ref{kernelsupport}, $\Lambda\otimes^{\L}_{\Lambda^{\op}\otimes \Lambda}\E^{\vee}\boxtimes \E\cong \Delta_{\ast}\R\underline{\Gamma_{\Z}}(\mathcal{O}_{\X})$.

\end{proof}
 \begin{proof}[Proof of Theorem \ref{gfkernel}]
   The method is that we compute $\Perf_{\Z}(\X)$ via special $\dg$ enhancement. Choose a compact generator $\E$ of $\D_{\Z,\qch}(\X)$. According to Lemma \ref{dgalge}, $\Perf(\Lambda)\cong \Perf_{\Z}(\X)$ where $\Lambda=\R\Hom(\E,\E)$. Therefore, the $\dg$ category $\Per_{\dg}(\Lambda)$ is a $\dg$ enhancement of $\Perf_{\Z}(\X)$. Thus, $$\HH_ {\bullet}(\Perf_{\Z}(\X))\cong \HH_{\bullet}(\Per_{\dg}(\Lambda))\cong \HH_{\bullet}(\Lambda).$$
   The second isomorphism is because the natural Yoneda embedding $\Lambda\rightarrow \Per_{\dg}(\Lambda)$ is a derived Morita equivalence.

   \par

   Since $\R\Gamma(\X,\E^{\vee}\otimes^{\L} \E)\cong \R\Hom(\E, \E)\cong \Lambda$, by induction (resolution of semi-free module), we have $\M\otimes^{\L} \N\cong \R\Gamma(\X,(\M\otimes^{\L} \E)\otimes^{\L} (\E^{\vee}\otimes^{\L} \N))$ for right $\Lambda$ module $\M$ and left $\Lambda$ module $\N$.
Similarly, $\R\Gamma(\X\times \X,(\E^{\vee}\boxtimes \E)\otimes^{\L} (\E^{\vee}\boxtimes \E))\cong \Lambda^{\op}\otimes \Lambda$. Hence in general, for $\Lambda^{\op}\otimes \Lambda$-module $\M$ and $\N$, we have$\colon$
$$\M\otimes^{\L}_{\Lambda^{\op}\otimes \Lambda} \N\cong \R\Gamma(\X\times \X, (\M\otimes \E^{\vee}\boxtimes \E)\otimes (\E^{\vee}\boxtimes \E\otimes \N)).$$

In particular, $\Lambda\otimes^{\L}_{\Lambda^{\op}\otimes \Lambda} \Lambda\cong \R\Gamma(\X\times \X, (\Lambda\otimes \E^{\vee}\boxtimes E)\otimes (\E^{\vee}\boxtimes \E)\otimes \Lambda)$.

According to Corollary \ref{Nkernel}, we have
$\HH_{\bullet}(\Lambda)\cong \Hom^{\bullet}(\mathcal{O}_{\X\times \X},\Delta_{\ast}\R\underline{\Gamma_{\Z}}(\mathcal{O}_{\X})\otimes^{\L}\Delta_{\ast}\R\underline{\Gamma_{\Z}}(\mathcal{O}_{\X})[\i])$
. The remaining thing is to compute $\Hom^{\ast}(\mathcal{O}_{\X\times \X},\Delta_{\ast}\R\underline{\Gamma_{\Z}}(\mathcal{O}_{\X})\otimes^{\L}\Delta_{\ast}\R\underline{\Gamma_{\Z}}(\mathcal{O}_{\X}))$.

\begin{center}
$$\xymatrix@C50pt@R20pt{&\X\ar[d]^{\Delta}\ar[ddl]_{\i\d}\ar[ddr]^{\i\d}& \\
  &\X\times \X\ar[dl]^{\p}\ar[dr]_{\q}& \\
  \X&&\X
  }$$
\end{center}

There is a triangle$\colon$
 $$\R\underline{\Gamma_{\Z}}\mathcal{O}_{\X}\rightarrow \mathcal{O}_{\X}\rightarrow \R\j_{\ast}\j^{\ast}\mathcal{O}_{\X}\rightarrow  \R\underline{\Gamma_{\Z}}\mathcal{O}_{\X}[1].$$
 Apply functor $\Delta{\ast}$, we have:
 $$\Delta_{\ast}\R\underline{\Gamma_{\Z}}\mathcal{O}_{\X}\rightarrow \Delta_{\ast}\mathcal{O}_{\X}\rightarrow \Delta_{\ast}\R\j_{\ast}\j^{\ast}\mathcal{O}_{\X}\rightarrow  \Delta_{\ast}\R\underline{\Gamma_{\Z}}\mathcal{O}_{\X}[1].$$
 Tensor (derived) with the object $\Delta_{\ast}\R\underline{\Gamma_{\Z}}\mathcal{O}_{\X}$ we get the triangle$\colon$
 $$\Delta_{\ast}\R\underline{\Gamma_{\Z}}\mathcal{O}_{\X}\otimes^{\L} \Delta_{\ast}\R\underline{\Gamma_{\Z}}\mathcal{O}_{\X}\rightarrow \Delta_{\ast}\mathcal{O}_{\X}\otimes^{\L} \Delta_{\ast}\R\underline{\Gamma_{\Z}}\mathcal{O}_{\X}\rightarrow \Delta_{\ast}\R\j_{\ast}\j^{\ast}\mathcal{O}_{\X}\otimes^{\L} \Delta_{\ast}\R\underline{\Gamma_{\Z}}\mathcal{O}_{\X}\rightarrow  +.$$
 By projection formula for closed immersion$\colon$
 $$\Delta_{\ast}\R\j_{\ast}\j^{\ast}\mathcal{O}_{\X}\otimes^{\L} \Delta_{\ast}\R\underline{\Gamma_{\Z}}\mathcal{O}_{\X}\cong \Delta_{\ast}(\R\j_{\ast}\j^{\ast}\mathcal{O}_{\X}\otimes^{\L} \L\Delta^{\ast}\Delta_{\ast}\R\underline{\Gamma_{\Z}}\mathcal{O}_{\X}).$$
 Since for any object $\F\in \D(\X)$$\colon$
 $$(\L\Delta^{\ast}\Delta_{\ast}\mathcal{O}_{\X})\otimes^{\L} \F\cong \L\Delta^{\ast}(\Delta_{\ast}\mathcal{O}_{\X}\otimes^{\L} \p^{\ast}\F).$$
$$\Delta_{\ast}\mathcal{O}_{\X}\otimes^{\L} \p^{\ast}\F\cong \Delta_{\ast}(\mathcal{O}_{\X}\otimes^{\L} \L\Delta^{\ast}\p^{\ast}\F)\cong \Delta_{\ast}\F.$$
Therefore
 $$(\L\Delta^{\ast}\Delta_{\ast}\mathcal{O}_{\X})\otimes^{\L} \F\cong \L\Delta^{\ast}\Delta_{\ast}\F.$$
 By result \cite[Thm 4.1]{CALDARARU200534}, $\L\Delta^{\ast}\Delta_{\ast}\mathcal{O}_{\X}\cong\bigoplus_{1\leq i\leq \n} \Omega^{\i}[\i]$.
 Hence $\L\Delta^{\ast}\Delta_{\ast}\R\underline{\Gamma_{\Z}}\mathcal{O}_{\X}$ supports in $\Z$. Again by the similar reason, $$\R\j_{\ast}\j^{\ast}\mathcal{O}_{\X}\otimes^{\L} \L\Delta^{\ast}\Delta_{\ast}\R\underline{\Gamma_{\Z}}\mathcal{O}_{\X}\cong 0.$$
 Therefore :
 $$\Delta_{\ast}\R\j_{\ast}\j^{\ast}\mathcal{O}_{\X}\otimes^{\L} \Delta_{\ast}\R\underline{\Gamma_{\Z}}\mathcal{O}_{\X}\cong 0.$$
 Hence
 $$\Delta_{\ast}\R\underline{\Gamma_{\Z}}\mathcal{O}_{\X}\otimes^{\L} \Delta_{\ast}\R\underline{\Gamma_{\Z}}\mathcal{O}_{\X}\cong \Delta_{\ast}\mathcal{O}_{\X}\otimes^{\L} \Delta_{\ast}\R\underline{\Gamma_{\Z}}\mathcal{O}_{\X}.$$
 Then
 $$\Hom^{\bullet}(\X\times \X,\Delta_{\ast}\R\underline{\Gamma_{\Z}}\mathcal{O}_{\X}\otimes^{\L} \Delta_{\ast}\R\underline{\Gamma_{\Z}}\mathcal{O}_{\X})\cong \Hom^{\bullet}(\X\times \X, \Delta_{\ast}\R\underline{\Gamma_{\Z}}\mathcal{O}_{\X}\otimes^{\L} \Delta_{\ast}\mathcal{O}_{\X}).$$
According to projection formula for closed immersion again$\colon$
 $$\Delta_{\ast}\R\underline{\Gamma_{\Z}}\mathcal{O}_{\X}\otimes^{\L} \Delta_{\ast}\mathcal{O}_{\X}\cong \Delta_{\ast}(\R\underline{\Gamma_{\Z}}\mathcal{O}_{\X}\otimes^{\L} \L\Delta^{\ast}\Delta_{\ast}\mathcal{O}_{\X}).$$
Hence
$$\Hom^{\bullet}(\X\times \X,\Delta_{\ast}\R\underline{\Gamma_{\Z}}\mathcal{O}_{\X}\otimes^{\L} \Delta_{\ast}\R\underline{\Gamma_{\Z}}\mathcal{O}_{\X})\cong \Hom^{\bullet}(\X, \R\underline{\Gamma_{\Z}}\mathcal{O}_{\X}\otimes \bigoplus_{1\leq \i\leq \n} \Omega^{\i}[\i]).$$
But derived tensor $\bigoplus_{1\leq i\leq \n} \Omega^{\i}[\i]$ with the triangle$\colon$
 $$\R\underline{\Gamma_{\Z}}\mathcal{O}_{\X}\rightarrow \mathcal{O}_{\X}\rightarrow \R\j_{\ast}\j^{\ast}\mathcal{O}_{\X}\rightarrow  \R\underline{\Gamma_{\Z}}\mathcal{O}_{\X}[1].$$
We have
$$\R\underline{\Gamma_{\Z}}\mathcal{O}_{\X}\otimes^{\L} \bigoplus_{1\leq \i\leq \n} \Omega^{\i}[\i]\rightarrow \mathcal{O}_{\X}\otimes^{\L} \bigoplus_{1\leq \i\leq \n} \Omega^{\i}[\i]\rightarrow \R\j_{\ast}\j^{\ast}\mathcal{O}_{\X}\otimes^{\L} \bigoplus_{1\leq \i\leq \n} \Omega^{\i}[\i]\rightarrow  \R\underline{\Gamma_{\Z}}\mathcal{O}_{\X}\otimes^{\L} \bigoplus_{1\leq \i\leq \n} \Omega^{\i}[\i][1].$$
Since
$$\R\j_{\ast}\j^{\ast}\mathcal{O}_{\X}\otimes^{\L} \bigoplus_{1\leq \i\leq \n} \Omega^{\i}[\i]\cong \R\j_{\ast}\j^{\ast}\bigoplus_{1\leq \i\leq \n} \Omega^{\i}[\i],$$
which is compatible with the triangle, we have
$$\R\underline{\Gamma_{\Z}}\mathcal{O}_{\X}\otimes \bigoplus_{1\leq \i\leq \n} \Omega^{\i}[\i]\cong \R\underline{\Gamma_{\Z}}\bigoplus_{1\leq \i\leq \n}\Omega^{\i}[\i].$$
 Finally
 $$\Hom^{\bullet}(\X\times \X,\Delta{\ast}\R\underline{\Gamma_{\Z}}\mathcal{O}_{\X}\otimes^{\L} \Delta{\ast}\R\underline{\Gamma_{\Z}}\mathcal{O}_{\X})\cong \R\Gamma_{\Z}(\X,\bigoplus_{1\leq \i\leq \n}\Omega^{\i}[\i])\cong \bigoplus_{\t}\bigoplus_{\p-\q=\t}\H_{\Z}^{\p}(\X,\Omega_{\X}^{\q})[-\t].$$
 \end{proof}

\vspace{5mm}

\section{Applications}
In this section, we apply the Hochschild homology theory of derived categories with support to certain problems of derived categories of variety and problems of semi-orthogonal decompositions.

\subsection{weak CY categories which are not admissible}
It is interesting to find some examples of weak Calabi-Yau categories which are not admissible subcategories of $\Perf(\Y)$, here $\Y$ is assumed to be projective smooth. Note that we say weak Calabi--Yau category if the Serre functor is a shifting. We obtain a series of examples as follows.

\begin{thm}\label{CY}
  Let $\X$ be a projective smooth variety of dimension $\n$. $\Z$ is a closed subscheme of $\X$.
  \begin{enumerate}
    \item If $\X$ is a Calabi-Yau variety, then $\Perf_{\Z}(\X)$ is a weak $\n$-Calabi-Yau category, and it is not an admissible subcategory of derived category of projective smooth varieties.
    \item If $\Z$ consist of points, then $\Perf_{\Z}(\X)$ is a weak $\n$-Calabi-Yau category, and it is not an admissible subcategory of derived category of projective smooth varieties.
  \end{enumerate}
\end{thm}

\begin{proof}
 the Serre functor of $\Perf_{\Z}(\X)$ is $\otimes \omega_{\X}[\n]$. For (1), since $\omega_{\X}\cong \mathcal{O}_{\X}$ , therefore the Serre functor of $\Perf_{\Z}(\X)$ is [$\n$]. For (2), since $\omega_{\X}$ is locally trivial around the points $\Z$, $\E\otimes \omega_{\X}\cong \E$ for any object $\E\in \Perf_{\Z}(\X)$. Thus, the Serre functor of $\Perf_{\Z}(\X)$ is [$\n$]. For both (1) and (2), we have that $\Perf_{\Z}(\X)$ admits Serre functor [$\n$]. To prove that $\Perf_{\Z}(\X)$ are not admissible subcategory of some derived category of projective smooth varieties, we need a lemma.

 \begin{lem}\label{Kuznetsovsod} (\cite[Corollary 5.4]{Kuznetsov_2019})
 Let $\mathcal{B}$ be an admissible subcategory of some derived category of a projective smooth vareity. Suppose it is a $\m$-Calabi--Yau category, then $\HH_{-\m}(\mathcal{B})\neq 0$.
 \end{lem}

 Back to the proof. Suppose $\Perf_{\Z}(\X)\cong\B$ is an admissible subcategory of $\Perf(\Y)$ where $\Y$ is a smooth projective variety. Since the
 Hochschild homology of $\B$ in paper \cite{Kuznetsov_2019} is compatible with some special $\dg$ enhancement of $\B:$ Choose a strong compact generator $\F$ of $\B$, let $\A1=\R\Hom(\F,\F)$ which is a $\dg$ algebra. Then, $\HH_{\bullet}(\A1)\cong \HH_{\bullet}(\B)$. Thus, $\HH_{-\n}(\Perf_{\Z}(\X))\cong \HH_{-\n}(\A1)\cong \HH_{-\n}(\B)\neq 0$.
 But according to Corollary \ref{-nHoch}, $\HH_{-\n}(\Perf_{\Z}(\X))=0$, a contradiction.
\end{proof}

\subsection{Applications to semi-orthogonal decomposition}
Let $\X$ be a projective smooth variety. It an interesting question that whether $\Perf(\X)$ have no nontrivial semi-orthogonal decompositions.
Kotaro Kawatani and Shinnosuke Okawa constructed an obstruction for the existence of nontrivial semi-orthogonal decompositions

\begin{prop}\label{sodsupport}(\cite[ Kotaro Kawatani and Shinnosuke Okawa, Thm 3.1]{article3})
  Let $\X$ be a proper smooth variety, suppose there is a nontrivial semi-orthogonal decomposition, $\Perf(\X)=\langle\mathcal{A},\mathcal{B}\rangle$. Then only one of the following cases
happen$\colon$\\
(1) The support of object in $\mathcal{A}$ is contained in $\Z=\2B\s\vert\omega_{\X}\vert$.\\
(2) The support of object in $\mathcal{B}$ is contained in $\Z=\2B\s\vert\omega_{\X}\vert$.\\
Furthermore, if (1) (or (2)) holds, then for $\x\in \X\setminus \Z$, $\k(\x)\in$ $\mathcal{B}$ (or $\mathcal{A}$).
\end{prop}

\begin{rem}
  This fact is generalized to para-canonical base locus $\Z=\PBs\vert\omega_{\X}\vert$ \cite{lin2021nonexistence}.
\end{rem}

It is natural to ask which component in Proposition \ref{sodsupport} support in the base locus of $\omega_{\X}$. We provide a criteria via Hochshcild homology.

\begin{thm}\label{A}
 \cite[Lemma 5.3]{pirozhkov2020admissible}Let $\X$ be a projective smooth variety of dimension $\n$. Suppose there is a nontrivial semi-orthogonal decomposition $\Perf(\X)=\langle\mathcal{A},\mathcal{B}\rangle$ with $\HH_{-\n}(\mathcal{B})\neq 0$. Then the support of any object in $\mathcal{A}$ is contained in $\Z=\2B\s\vert\omega_{\X}\vert$. The skyscraper sheaves $\k(\x)$ with $\x\in \X\setminus \Z$ belongs to $\mathcal{B}$. Furthermore, $\HH_{-\n}(\mathcal{A})= 0$, $\HH_{-\n}(\Perf(\X))\cong \HH_{-\n}(\mathcal{B})$. If $\mathcal{B}$ is a $\n$ Calabi-Yau category, then it is indecomposable.
\end{thm}

\begin{rem}
  $\mathcal{A}$ should be regarded as the smaller piece in the components of the semi-orthogonal decomposition.
\end{rem}

Before proving Theorem \ref{A}, we need some preparations that Hochschild homology is additive with respect to semi-orthogonal decompositions. Again, it is well known for experts, we provide proof here for completion.

\begin{thm}\label{soddg}
  Let $\D_{\dg}$ be a small $\dg$ category with homotopic category $[\D_{\dg}]=\D$. Suppose there is a semi-orthogonal decomposition of $\D:=\langle \1A,\B\rangle$. Let $\1A_{\dg}$ and $\B_{\dg}$ be the full sub $\dg$ category of $\D_{\dg}$ such that objects belong to $\1A$ and $\B$ respectively. Then, $$\HH_{\bullet}(\D_{\dg})\cong \HH_{\bullet}(\1A_{\dg})\oplus \HH_{\bullet}(\B_{\dg}).$$
\end{thm}

\begin{proof}
  A part of proof here is taken from G.\ Tabuada's book \cite[Proposition 2.2]{NM}. Consider the $\dg$ bi-moudle
    $$\xymatrix{\M\colon \B_{\dg}^{op}\otimes \1A_{\dg}\rightarrow \C_{\dg}(\k)& (\b,\a)\mapsto \D_{\dg}(\X)(\a,\b)}$$
Define the gluing of $\dg$ categories with respect to $\M$ as
    $$\T(\1A_{\dg},\B_{\dg};\M)(\x,\y)=
    \begin{cases}
      \1A_{\dg}(\x,\y), & \mbox{if } \x,\ \y \in \1A_{\dg} \\
      \B_{\dg}(\x,\y), & \mbox{if } \x,\ \y  \in \B_{\dg}\\
      \M(\x,\y), & \mbox{if } \x \in \1A_{\dg},\ \y\in \B_{\dg} \\
      0, & \mbox{otherwise},\ \x\in \B_{\dg},\ \y\in \1A_{\dg}.
    \end{cases}$$
    The natural morphism of $\dg$ categories
  $$\T\rightarrow\D_{\dg}$$
  is a derived Morita equivalence. We have natural isomorphism of Hochschild homology
  $$\HH_{\bullet}(\T)\cong \HH_{\bullet}(\D_{\dg}).$$
  There is a diagram of $\dg$ categories

  \begin{center}
  $$\xymatrix{\B_{\dg}\ar[r]_{\i_{2}} &\T\ar[r]_{\L}\ar@/_/[l]_{\R}&\1A_{\dg}\ar@/_/[l]_{\i_{1}}
}.$$
  \end{center}

   It induces a commutative diagram of triangles of Hochschild complex \\
  \centerline{\xymatrix{\C(\B_{\dg})\ar[r]^{\i_{2}}\ar[d]^{\i\d}&\C(\T)\ar[r]^{\L}\ar[d]^{\R+\L}&\C(\1A_{\dg})\ar[d]^{\i\d}\ar[r]&\C(\B_{\dg})[1]\ar[d]\\
\C(\B_{\dg})\ar[r]&\C(\B_{\dg})\oplus \C(\1A_{\dg})\ar[r]&\C(\1A_{\dg})\ar[r]&\C(\B_{\dg})[1])
}}
It induces an isomorphism of Hochschild complex $\C(\T)\cong \C(\1A_{\dg})\oplus\C(\B_{\dg})$. Thus $$\HH_{\bullet}(\D_{\dg})\cong \HH_{\bullet}(\T)\cong \HH_{\bullet}(\1A_{\dg})\oplus \HH_{\bullet}(\B_{\dg}).$$
\end{proof}

 \begin{proof}[Proof of Theorem \ref{A}]
 By $\H\K\R$ and the additive theory of Hochschild homology Theorem \cite{K}, $\H^{0}(\X,\omega_{\X})\neq 0$.
 Hence if the nontrivial semi-orthogonal decomposition exists, then the canonical bundle is not trivial,
 and the base locus of the canonical bundle is a proper closed subset of $\X$.

 \par

 We write $\Per_{\dg}(\X)$ as the natural $\dg$ enhancement of $\Perf(\X)$. The objects are $\K$-injective perfect complexes. Define $\Per_{\Z,\dg}(\X)$ as the full $\dg$ subcategory of $\Per_{\dg}(\X)$ whose objects support in $\Z$. We write $\1A_{\dg}$ and $\B_{\dg}$ as the natural $\dg$ enhancement corresponding to $\Per_{\dg}(\X)$.

 \par

 According to Proposition \ref{sodsupport}, either all objects of $\1A$ support in $\Z$, or all objects of $\B$ support in $\Z$. Suppose all objects of $\B$ support in $\Z$, then the original semi-orthogonal decomposition induces a semi-orthogonal decomposition
 $$\Perf_{\Z}(\X):=\langle \1A_{\Z},\B\rangle$$
 where $\1A_{\Z}$ is the full triangulated subcategory of $\1A$ whose objects support in $\Z$. First, given any object $\E\in \Perf_{\Z}(\X)$, there is a triangle with respect to the original semi-orthogonal decomposition
 $$\E_{1}\rightarrow \E\rightarrow \E_{2}\rightarrow \E_{1}[1].$$
 That is, $\E_{2}\in \1A$ and $\E_{1}\in \B$.
 Furthermore, support of $\E$ and $\E_{2}$ are in $\Z$. Thus, using Lemma \ref{support}, $\E_{1}$ also supports in $\Z$. Hence $\E_{1}\in \1A_{\Z}$.
 Second, $\Hom(\B,\1A_{\Z})=0$ by the original semi-orthogonal decomposition.

 \par

 According to Theorem \ref{soddg}, we have isomorphism of Hochschild homology
 $$\HH_{\bullet}(\Per_{\Z,\dg}(\X))\cong \HH_{\bullet}(\1A_{\Z,\dg})\oplus\HH_{\bullet}(\B_{\dg}).$$
 Now, with the same technique in proof of Theorem \ref{CY}, the Hochschild Homology $\HH_{\bullet}(\B)$ is isomorphic to Hochschild homology of a special $\dg$ enhancement of $\B\colon$
 Choosing a strong compact generator $\F$ of $\B$ (assume it is $\K$-injective), define $\dg$ algebra $\Lambda=\Hom_{\dg}(\F,\F)$. Then $\HH_{\bullet}(\Lambda)\cong \HH_{\bullet}(\B)$. We regard the $\dg$ algebra $\Lambda$ as a $\dg$ category with one object $\circ$. There is a natural morphism of $\dg$ categories
 $$\Lambda\rightarrow \B_{\dg}\quad\quad \circ\mapsto \F.$$
 It is a derived Morita morphism since this natural map induce natural equivalence $\D(\Lambda)\cong \D(\B_{\dg})\cong \B$. Thus, $\HH_{\bullet}(\Lambda)\cong \HH_{\bullet}(\B_{\dg})$. Finally, $\HH_{-\n}(\B_{\dg})\cong \HH_{-\n}(\B)\neq 0$ should be a subspace of $\HH_{-\n}(\Perf_{\Z}(\X))=0$, a contradiction. So, we prove that every objects of $\1A$ support in $\Z$.

 \par

 We prove that in this case, $\HH_{-\n}(\1A)=0$. Otherwise, the same method shows that every objects of $\B$ should support in $\Z$, a contradiction. It is easy to see that we can exchange the role of $\1A$ and $\B$.

 \par

 At last, we prove $\mathcal{B}$ is indecomposable if it is $\n$ Calabi-Yau category. If not, let $\mathcal{B}=\langle \mathcal{B}_{1},\mathcal{B}_{2}\rangle$. Clearly, $\mathcal{B}_{1}$ and $\mathcal{B}_{2}$ are both $\n$-Calabi-Yau categories. There is a semi-orthogonal decomposition
 $$\Perf(\X)=\langle \mathcal{A},\mathcal{B}_{1},\mathcal{B}_{2}\rangle$$
  The same argument above shows that $\k(\x)\in \mathcal{B}_{2}$ and
  $\k(\x)\in \langle \mathcal{A}, \mathcal{B}_{1}\rangle$ for any closed points $\x\in \X\setminus \Z$, a contradiction.
 \end{proof}

\begin{eg}
 Let $\Y$ be a Calabi--Yau variety of dimension $\n$, and $\X$ be the blow-up of one ponit
 of $\Y$. According to Orlov blow-up formula, there is a semi-orthogonal decomposition,
 $$ \Perf(\X)=\langle\mathcal{O}_{\E}(\E),\Perf(\Y)\rangle.$$
 $\E$ is the exceptional locus of the blow-up $\f\colon \X\rightarrow \Y$.
 The base locus of $\omega_{\X}$ is $\E$, and clearly support of $\langle \mathcal{O}_{\E}(\E)\rangle \subseteq \E$. Point sheaf $\k(\x)\in \Perf(\Y)$ for $\x\in \X\setminus \E\cong \Y\setminus \p\t $. Furthermore, $\HH_{-\n}(\langle\mathcal{O}_{\E}(\E)\rangle)\cong \HH_{-\n}(\mathsf{D}^{\mathsf{b}}(\p\t))= 0$.
 \end{eg}

\begin{rem}\label{remA}
From the above's proof, the $-\n$ Hochschild homology of $\1A$ vanishes.
It coincides with our intuition that $\1A$ should be smaller. In addition, the $-\n$ Hochschild homology of $\1A$ and $\B$ can't be nonzero at the same time. The situation $\Perf(\X)=\langle \mathsf{D}^{\mathsf{b}}(\mathsf{Y}_{1}), \mathsf{D}^{\mathsf{b}}(\mathsf{Y}_{2}), \cdots \rangle$ can't happen that $\dim \X= \dim \Y_{1}= \dim \Y_{2}$ and both $\Y_{1}$, $\Y_{2}$ are varieties with canonical bundles having non zero global sections.
\end{rem}

Actually, we can also only use Theorem \ref{gfkernel} to prove Theorem \ref{A} since the following theorem.

\begin{thm}\label{refinesod}
 Assume there is a non trivial semi-orthogonal decomposition, $\Perf(\X)=\langle\mathcal{A}, \mathcal{B}\rangle.$ Denote $\Z=\2B\s\vert \omega_{\X}\vert$. It could happen that $\Z$ have zero dimensional components. Deleting zero component of $\Z$, denote it $\Z'$, then the support of $\mathcal{A}$ is in $\Z'$ or the support of $\mathcal{B}$ is in $\Z'$.
 \end{thm}

 \begin{proof}
 According to the Proposition \ref{sodsupport}, the support of  $\mathcal{A}$ or the support of $\mathcal{B}$ is in $\Z$. Without loss of generality, assume the support of $\mathcal{A}$ is in $\Z$. We prove that it is actually in $\Z'$. Suppose $\exists$ $\E\in \mathcal{A}$ whose support contains some point components, then it is easy to show that $\mathcal{A}_{\Z\setminus \Z'}$  is nontrivial. Then we have an orthogonal decomposition $\mathcal{A}=\langle\mathcal{A}_{\Z\setminus \Z'},\mathcal{A}_{\Z'}\rangle$. Therefore we have a nontrivial semi-orthogonal decomposition, $\Perf(\X)=\langle \mathcal{A}_{\Z\setminus \Z'}, \mathcal{C}\rangle$ where $\mathcal{C}\colon=\langle \mathcal{A}_{Z'}, \mathcal{B}\rangle.$ However, $\mathcal{A}_{\Z\setminus \Z'}\otimes \omega_{\X} =\mathcal{A}_{\Z\setminus \Z'}$, hence it is a nontrivial orthogonal decomposition, contradicts that $\X$ is connected.
 \end{proof}

 \subsection{Applications to a Kuznetsov's conjecture}
 Kunetsov problem \cite{Kuznetsov_2016}$\colon$ Let $\X$ be a projective smooth variety, $\mathcal{B}$ be a $\n$-Calabi--Yau category such that $\n>\dim \X$. Then $\mathcal{B}$ can not be an admissible subcategory of $\Perf(\X)$ since otherwise $0= \HH_{-\n}(\X)\geq \HH_{-\n}(\mathcal{B})\neq 0$. Hence in order to find $\n$-Calabi--Yau subcategory inside $\Perf(\X)$, we must have $\n\leq \dim \X$. It was conjectured that the case that $\n=\dim \X$ should be a certain boundary condition, and that the semi-orthogonal decomposition
only comes from Orlov blow-up formula.

\begin{conj}\label{Kuznetsovconj}\cite{Kuznetsov_2019}
  Suppose we have $\n=\dim \X$ Calabi--Yau category $\mathcal{B}$ $\subset$ $\Perf(\X)$, then $\exists$ a projective smooth variety $\X'$ with morphism
$\X\rightarrow \X'$ being the blow-up. Furthermore $\mathcal{B} \cong \mathsf{D}^{\mathsf{b}}(\mathsf{X}')$.

\end{conj}

 \par

  However, Kuznetsov's conjecture is far from being true, see ``Compact $Hyper-K\ddot{a}hler$ Categories'' \cite[Section 5]{rol2015compact}. It was shown in the paper that there are infinite many geometric $4$-folds containing $4$-Calabi--Yau (connected) categories which are not derived category of a projective smooth variety. The problem still make sense when we firstly assume the Calabi--Yau category to be derived category of variety, that is, Kuznetsov problem can be refined as follows$\colon$

 \begin{conj}\label{conjCY}
 Suppose we have a proper fully faithful embedding
 $\Perf(\Y)\hookrightarrow \Perf(\X)$, $\dim \X= \dim \Y$, $\Y$ is a Calabi--Yau variety. Then, one can find a Calabi--Yau variety $\Y'$ such that $\X$ is a certain blow-up of $\Y'$. Furthermore, $\Perf(\Y')\cong \Perf(\Y)$.
 \end{conj}

\begin{rem}
  We can not expect that $\Y'\cong \Y$, since there may be nonbirational Fourier-Mukai partners of some Calabi--Yau varieties, for example, the general $\K_{3}$ surfaces.
 \end{rem}

 \begin{thm}\label{B}
 Let $\X$ be a projective smooth variety of dimension $\n$. Suppose $\dim \H^{0}(\X,\omega_{\X}) \geq 2$, then any $\n$-Calabi--Yau admissible subcategory of $\Perf(\X)$ is not equivalent to a derived category of a smooth projective variety.
 \end{thm}

\begin{proof}
  Suppose there is a $\n$-Calabi--Yau subcategory $\mathcal{B}\subset \Perf(\X)$. According to Theorem \ref{A}, $\dim\HH_{-\n}(\mathcal{B})= \dim \HH_{-\n}(\DbX)\geq 2$. Thus $\mathcal{B}$ is not connected (An admissible subcategory is connected if $\HH^{0}(\bullet)\cong \k$.). Suppose $\mathcal{B}$ is geometric, that is $\mathcal{B}\cong \Perf(\Y)$ for some Calabi--Yau variety $\Y$. Since $\mathcal{B}$ is not connected, Y is not connected too, a contradiction.
  \end{proof}

\begin{conj}\label{ConjCy}
 Let $\mathcal{B}$ be an admissible $\m$-$\C\Y$ subcategory of $\DbY$, where $\Y$ is a projective smooth variety, and $\m\geq 0$. Then, $\mathcal{B}$ is indecomposable if and only if $\HH^{0}(\mathcal{B})\cong \k$.
\end{conj}

\begin{rem}
  If the Conjecture \ref{ConjCy} is true, then Theorem \ref{B} implies that there is no $\n$-Calabi-Yau subcategory in $\Perf(\X)$ ($\dim\X=\n$) if $\dim \H^{0}(\X,\omega_{\X}) \geq 2$, which also satisfies Kuznetsov's original intuition.
\end{rem}

  The following theorem shows some intuition about this A.\ Kuznetsov's problem$\colon$

 \begin{thm}\label{C}
 Let $\X$ be a projective smooth variety. Suppose its derived category $\Perf(\X)$ admits a
 nontrivial semi-orthogonal decomposition with component $\mathcal{B}$ being $\n$-$\C\Y$ category. Here $\n$ is dimension of $\X$. Then, the base locus $\Z$ of the canonical bundle is a proper closed subset. Furthermore, any point sheaves $\k(\x)$ with $\x\in \X\setminus \Z$ belongs to $\mathcal{B}$.
 \end{thm}

\begin{proof}
This follows from Theorem \ref{A} since $\HH_{-\n}(\mathcal{B})\neq 0$ for Calabi-Yau category $\mathcal{B}$.
\end{proof}

 Since the good knowledge of the birational geometry of the surfaces, we expect that the Kuznetsov's problem can be achieved for the cases of surfaces. Since the Theorem \ref{B}, we always assume the $2$-Calabi-Yau category to be connected.

\begin{enumerate}
  \item $\k(\X)=-\infty$. If $\X$ is a rational surface, then by Hochschild homology techniques, there is no $\C\Y$-$2$ category in $\Perf(\X)$. If $\X$ is ruled surface, write $\X=\mathbb{P}^{1}\times \C$ for some curves $\C$. Then again by Hochschild homology theory, $\Perf(\X)$ has no admissible $\C\Y$-$2$ category.
  \item $\k(\X)=0$, it is well known that the Enrique surfaces (and their blow-ups) or bi-elliptic surfaces (and their blow-ups) is of geometric genus $0$, hence there is no admissible $\C\Y$-$2$ in $\DbX$.
According to classical classification of projective smooth surfaces, it remains to the cases: blow-up of abelian surfaces, $\K_{3}$ surfaces, elliptic surfaces, and minimal surfaces of general type.
\end{enumerate}

\par

Finally, we propose an interesting strategy to refine this Kuznetsov's  problem. We need a lemma which is interesting in itself.

\begin{lem}\label{lem 4.19}
Let $\T$ be a triangulated category with a semi-orthogonal decomposition
 $\mathcal{T}=\langle \mathcal{A}, \mathcal{B}\rangle$. Suppose there is a fully faithful embedding of triangulated categories $\i\colon \mathcal{T_{1}}\hookrightarrow \mathcal{T}$ with right adjoint $\R$, and $\mathcal{T_{1}}$ admits Serre functor. Let $\{\E_{\j}[\m]\}$ be a set of objects that is a spanning class of $\mathcal{T_{1}}$. Assume $\i$ maps this set of objects into $\mathcal{B}$, then $\i$ maps $\mathcal{T_{1}}$ into
 $\mathcal{B}$. Note that the statement is true if replacing $\mathcal{B}$ with $\mathcal{A}$, and right adjoint with left adjoint.
\end{lem}

\begin{proof}
 Write $\{\E_{\i}[\m]\}_{\i\in \I}$ as a set of spanning class of $\mathcal{T_{1}}$. Since the category $\mathcal{T}_{\infty}$ has Serre functor, we can assume the spanning class as the left spanning class. Let $\E\in \mathcal{T_{1}}$, and any $\a\in \mathcal{A}$. By assumption, $\Hom(\i(\E_{\j}[\m]), \mathcal{A})=0$ for any integers $\m$ and $\j$, then $\Hom(\i(\E_{\j}[\m]),\a)\cong \Hom(\E_{\j}[\m],\R(\a))=0$. Since $\{\E_{\i}[\m]\}_{\i\in \I}$ is a spanning class, we have $\R(\a)\cong 0$. But $\Hom(\i(\E),\a)\cong \Hom(\E,\R(\a))\cong 0$, hence $\i(\E) \in \mathcal{B}$.
 \end{proof}

  Let $\X$ be the blow-up of $\Y$ over points $\{\y_{\i}\}_{\i}$. We write the blow-up morphism as $\f\colon \X\longrightarrow \Y$. Let $\{\C_{\i}\}$ be the exceptional divisors which are contracted to the points $\y_{\i}$. We get an embedding $\L\f^{\ast}\colon \Perf(\Y)\hookrightarrow \Perf(\X)$. Define the distinguished objects $\D_{\i}$ attached to the contraction $\f$ as $\L\f^{\ast}\k(\y_{\i})$. Note that the distinguished object $\D_{\i}$ supports at the exceptional divisor $\C_{\i}$.

 \begin{thm}\label{D}
  Let $\X$ be the blow-up of dimension $\n$ Calabi-Yau varity $\Y$. Assume there is an element $\phi$ of $\mathsf{Aut}(\Perf(\X))$ which maps all distinguished objects $\D_{\i}$ to the one whose support is disjoint to all exceptional divisors (base locus of $\K_{\X}$). Then, let $\mathcal{B}$ be any admissible $\C\Y$-$\n$ of $\Perf(\X)$ , we have $\D_{\i}\in \mathcal{B}$.
 \end{thm}

\begin{proof}
 Suppose there is a $\C\Y$-$\n$ $\mathcal{B}$ in $\Perf(\X)$. Consider the semi-orhtogonal decomposition $\Perf(\X)=\langle \mathcal{B}^{\perp},\mathcal{B}\rangle$ and the
 embedding $\Perf(\Y)\hookrightarrow \Perf(\X)$. After an auto-equivalence of the derived category $\Perf(\X)$, we get a new semi-orthogonal decomposition with a $\C\Y$-$\n$ component $\mathcal{B}'\cong \mathcal{B}$. We see that the embedding $\L\f^{\ast}$ composing with the auto-equivalence $\phi$ maps the distinguish objects to $\mathcal{B}'$. The reason is as follows: $\phi\D_{\i}$ belongs to $\mathcal{B}'$ or $\mathcal{B}'^{\perp}$. According to Theorem \ref{A}, any objects in $\mathcal{B}'^{\perp}$ must support in the exceptional divisors $\cup\C_{\i}$, hence we have $\phi\D_{\i}\in \mathcal{B}'$. Thus, there exist $\D'_{\i}\in \mathcal{B}'$ such that $\phi \D_{\i}\cong\D'_{\i}$, and then $\D_{\i}=\phi^{-1}\D'_{\i}$ which belongs to $\mathcal{B}$.
 \end{proof}

 \begin{rem} \
\begin{enumerate}
  \item Unfortunately, the usually automorphism of varieties, twisted by line bundles, and shifting do not satisfy our assumption in Theorem \ref{D}. It is interesting to the author that if there is a nontrivial auto-equivalence satisfying assumption in Theorem \ref{D}.
  \item The assumption can't not happen for surfaces. This follows easily from Proposition 12.15 \cite{book2} which is originally proved by Kawamata: There is a correspondence $\Gamma$ which induces automorphism of $\X$. The support of objects $\phi(\D_{\i})$ under the auto-equivalence $\phi$ must contain exceptional curve $\C_{\i}$. It will still be interesting if the assumption could happen for higher dimensional varieties.
\end{enumerate}
\end{rem}

 \begin{thm}\label{Main1}
  Let $\Y$ be a smooth projective Calabi-Yau variety of dimension $\n$, $\f:\X\rightarrow \Y$ is the bolw-up of $\Y$ over points $\{\y_{\i}\}$. Define the distinguish objects
  $\D_{\i}=\L\f^{\ast}\k(\y_{\i})$. Let $\mathcal{B}$ be an $\n$ Calabi-Yau admissible subcategory of $\Perf(\X)$. If all $\D_{\i}\in \mathcal{B}$, then $\mathcal{B}=\L\f^{\ast}\Perf(\Y)$.
\end{thm}

 \begin{proof}
 $$\Perf(\Y)\hookrightarrow\Perf(\X)=\langle \mathcal{B}{^\perp},\mathcal{B}\rangle $$
 \par
 Since the skyscraper sheaves of all the closed points (and shifting) form a spanning class, according to Theorem \ref{D} and Theorem \ref{A}, $\L\f^{\ast}\k(\y_{\i})\in \mathcal{B}$, and then by Lemma \ref{lem 4.19}, we get an embedding $\Perf(\Y)\hookrightarrow \mathcal{B}$.
 The embedding $\Perf(\Y)\hookrightarrow \mathcal{B}$ must induce an orthogonal decomposition because $\mathcal{B}$ is a Calabi-Yau category.
 $$\mathcal{B}=\langle \Perf(\Y)^{\perp},\Perf(\Y)\rangle.$$
 Applying additive theory of Hochschild cohomology for orthogonal decompositon, see \cite[Proposition 5.5]{K}, we have
 $$\HH^{0}(\mathcal{B})\cong \HH^{0}(\Perf(\Y))\oplus \HH^{0}(\Perf(\Y)^{\perp}).$$
 If we assume $\Perf(\Y)^{\perp}$ to be nontrivial, then $\HH^{0}(\Perf(\Y)^{\perp})\neq 0$, and obviously $\HH^{0}(\Perf(\Y))\neq 0$. However, by Theorem \ref{A}, we have $\HH^{0}(\mathcal{B})=\HH^{0}(\Y)=\k$, a contradiction. Thus, the embedding $\Perf(\Y)\hookrightarrow\mathcal{B}$ is an equivalence.
\end{proof}


\begin{thebibliography}{99}

\bibitem[1]{KELLER1998223}
Bernhard Keller.
\newblock Invariance and localization for cyclic homology of dg algebras.
\newblock {\em Journal of Pure and Applied Algebra}, 123(1):223 -- 273,
1998.

\bibitem[2]{KELLER19991}
Bernhard Keller.
\newblock On the cyclic homology of exact categories.
\newblock {\em Journal of Pure and Applied Algebra}, 136(1):1 -- 56, 1999.

\bibitem[3]{article}
Bertrand Toën.
\newblock The homotopy theory of dg-categories and derived morita theory.
\newblock {\em Inventiones mathematicae}, 167:615--667, 03 2007.

\bibitem[4]{D}
Vladimir Drinfeld.
\newblock Dg quotients of dg categories.
\newblock {\em J. Algebra}, page 210114, 2008.

\bibitem[5]{2010JAMS...23..853L}
Valery~A. {Lunts} and Dmitri~O. {Orlov}.
\newblock {Uniqueness of enhancement for triangulated categories}.
\newblock {\em Journal of the American Mathematical Society},
23(3):853--908,
  Jul 2010.

\bibitem[6]{book2}
D.Huybrechts.
\newblock {\em Fourier–Mukai transforms in algebraic geometry}.
\newblock Clarendon Press • Orxord, Great Clarendon Street, Oxford OX2 6DP,
  2006.

\bibitem[7]{K}
Alexander {Kuznetsov}.
\newblock {Hochschild homology and semiorthogonal decompositions}.
\newblock {\em arXiv e-prints}, page arXiv:0904.4330, Apr 2009.

\bibitem[8]{1998math......2007K}
Bernhard {Keller}.
\newblock {On the cyclic homology of ringed spaces and schemes}.
\newblock {\em arXiv Mathematics e-prints}, page math/9802007, Feb 1998.

\bibitem[9]{CM_1988__65_2_121_0}
N.~Spaltenstein.
\newblock Resolutions of unbounded complexes.
\newblock {\em Compositio Mathematica}, 65(2):121--154, 1988.

\bibitem[10]{orlov1996equivalences}
Dmitri Orlov.
\newblock Equivalences of derived categories and k3 surfaces, 1996.

\bibitem[11]{bondal2002generators}
Alexei Bondal and Michel~Van den Bergh.
\newblock Generators and representability of functors in commutative and
  noncommutative geometry, 2002.

\bibitem[12]{rouquier_2008}
Raphaël Rouquier.
\newblock Dimensions of triangulated categories.
\newblock {\em Journal of K-theory: K-theory and its Applications to
Algebra,
  Geometry, and Topology}, 1(2):193–256, 2008.

\bibitem[13]{ASENS_1992_4_25_5_547_0}
Amnon Neeman.
\newblock The connection between the $k
$-theory localization theorem of
  thomason, trobaugh and yao and the smashing subcategories of bousfield and
  ravenel.
\newblock {\em Annales scientifiques de l'\'Ecole Normale Sup\'erieure},
Ser.
  4, 25(5):547--566, 1992.

\bibitem[14]{stacks-project}
The {Stacks Project Authors}.
\newblock \textit{Stacks Project}.
\newblock \url{https://stacks.math.columbia.edu}, 2019.

\bibitem[15]{canonaco_stellari_2014}
Alberto Canonaco and Paolo Stellari.
\newblock Fourier–Mukai functors in the supported case.
\newblock {\em Compositio Mathematica}, 150(8):1349–1383, 2014.

\bibitem[16]{CALDARARU200534}
Andrei~Caldararu.
 \newblock The Mukai pairing—ii: the Hochschild–Kostant–Rosenberg
  isomorphism.
 \newblock {\em Advances in Mathematics}, 194(1):34 -- 66, 2005.

\bibitem[17]{article3}
Kotaro {Kawatani} and Shinnosuke {Okawa}.
\newblock {Nonexistence of semiorthogonal decompositions and sections of the
  canonical bundle}.
\newblock {\em arXiv e-prints}, page arXiv:1508.00682, Aug 2015.

\bibitem[18]{rol2015compact}
Roland Abuaf and Grégoire Menet.
\newblock Compact hyper-kähler categories, 2015.

\bibitem[19]{Kuznetsov_2016}
Alexander Kuznetsov.
\newblock Derived categories view on rationality problems.
\newblock {\em Rationality Problems in Algebraic Geometry}, page 67–104,
  2016.
\bibitem[20]{Canonaco_2017}
Alberto Canonaco and Paolo Stellari.
\newblock A tour about existence and uniqueness of dg enhancements and
lifts.
\newblock {\em Journal of Geometry and Physics}, 122:28–52, Dec 2017.

\bibitem[21]{keller2006differential}
Bernhard Keller.
\newblock On differential graded categories, 2006.

\bibitem[22]{tabuada2004une}
Goncalo Tabuada.
\newblock Une structure de categorie de modeles de quillen sur la categorie
des
  dg-categories, 2004.

\bibitem[23]{alex2014semiorthogonal}
\newblock {Semiorthogonal decompositions in algebraic geometry, Alexander
Kuznetsov, 2014}
\newblock {arXiv math.AG 1404.3143}

\bibitem[24]{Kuznetsov_2019}
Alexander Kuznetsov.
\newblock Calabi–Yau and fractional Calabi–Yau categories.
\newblock {\em Journal für die reine und angewandte Mathematik (Crelles
  Journal) 2019(753):239–267, Aug 2019.}

\bibitem[25]{Antieau2018OnTU}
Benjamin Antieau.
\newblock On the uniqueness of infinity-categorical enhancements of
  triangulated categories.
\newblock {\em arXiv: Algebraic Geometry}, 2018.

\bibitem[26]{NM}
Goncalo Tabuada.
\newblock Noncommutative motives, University Lecture Series 63, American
Mathematical Society, 2015. DOI:
https://doi.org/http://dx.doi.org/10.1090/ulect/063.

\bibitem[27]{Orlov2016SMOOTHAP}
D.~Orlov.
\newblock Smooth and proper noncommutative schemes and gluing of dg
categories.
\newblock {\em Advances in Mathematics}, 302:59--105, 2016.

\bibitem[28]{tabuada2019noncommutative}
Goncalo Tabuada.
\newblock Noncommutative counterparts of celebrated conjectures.
\newblock arXiv:1812.08774[math.AG].

\bibitem[29]{pirozhkov2020admissible}
Dmitrii Pirozhkov.
\newblock Admissible subcategories of del Pezzo surfaces,
arXiv:2006.07643[math.AG].

\bibitem[30]{lin2021nonexistence}
Xun Lin.
\newblock On nonexistence of semi-orthogonal decompositions in algebraic
geometry.
\newblock  arXiv:2107.09564[math.AG], 2021.


\begin{thebibliography}{99}

\bibitem[1]{KELLER1998223}
Bernhard Keller.
\newblock Invariance and localization for cyclic homology of dg algebras.
\newblock {\em Journal of Pure and Applied Algebra}, 123(1):223 -- 273, 1998.

\bibitem[2]{KELLER19991}
Bernhard Keller.
\newblock On the cyclic homology of exact categories.
\newblock {\em Journal of Pure and Applied Algebra}, 136(1):1 -- 56, 1999.

\bibitem[3]{article}
Bertrand Toën.
\newblock The homotopy theory of dg-categories and derived morita theory.
\newblock {\em Inventiones mathematicae}, 167:615--667, 03 2007.

\bibitem[4]{D}
Vladimir Drinfeld.
\newblock Dg quotients of dg categories.
\newblock {\em J. Algebra}, page 210114, 2008.

\bibitem[5]{2010JAMS...23..853L}
Valery~A. {Lunts} and Dmitri~O. {Orlov}.
\newblock {Uniqueness of enhancement for triangulated categories}.
\newblock {\em Journal of the American Mathematical Society}, 23(3):853--908,
  Jul 2010.

\bibitem[6]{book2}
D.Huybrechts.
\newblock {\em Fourier–Mukai transforms in algebraic geometry}.
\newblock Clarendon Press • Orxord, Great Clarendon Street, Oxford OX2 6DP,
  2006.

\bibitem[7]{K}
Alexander {Kuznetsov}.
\newblock {Hochschild homology and semiorthogonal decompositions}.
\newblock {\em arXiv e-prints}, page arXiv:0904.4330, Apr 2009.

\bibitem[8]{1998math......2007K}
Bernhard {Keller}.
\newblock {On the cyclic homology of ringed spaces and schemes}.
\newblock {\em arXiv Mathematics e-prints}, page math/9802007, Feb 1998.

\bibitem[9]{CM_1988__65_2_121_0}
N.~Spaltenstein.
\newblock Resolutions of unbounded complexes.
\newblock {\em Compositio Mathematica}, 65(2):121--154, 1988.

\bibitem[10]{orlov1996equivalences}
Dmitri Orlov.
\newblock Equivalences of derived categories and k3 surfaces, 1996.

\bibitem[11]{bondal2002generators}
Alexei Bondal and Michel~Van den Bergh.
\newblock Generators and representability of functors in commutative and
  noncommutative geometry, 2002.

\bibitem[12]{rouquier_2008}
Raphaël Rouquier.
\newblock Dimensions of triangulated categories.
\newblock {\em Journal of K-theory: K-theory and its Applications to Algebra,
  Geometry, and Topology}, 1(2):193–256, 2008.

\bibitem[13]{ASENS_1992_4_25_5_547_0}
Amnon Neeman.
\newblock The connection between the $k
$-theory localization theorem of
  thomason, trobaugh and yao and the smashing subcategories of bousfield and
  ravenel.
\newblock {\em Annales scientifiques de l'\'Ecole Normale Sup\'erieure}, Ser.
  4, 25(5):547--566, 1992.

\bibitem[14]{stacks-project}
The {Stacks Project Authors}.
\newblock \textit{Stacks Project}.
\newblock \url{https://stacks.math.columbia.edu}, 2019.

\bibitem[15]{canonaco_stellari_2014}
Alberto Canonaco and Paolo Stellari.
\newblock Fourier–Mukai functors in the supported case.
\newblock {\em Compositio Mathematica}, 150(8):1349–1383, 2014.

\bibitem[16]{CALDARARU200534}
Andrei~Caldararu.
 \newblock The Mukai pairing—ii: the Hochschild–Kostant–Rosenberg
  isomorphism.
 \newblock {\em Advances in Mathematics}, 194(1):34 -- 66, 2005.

\bibitem[17]{article3}
Kotaro {Kawatani} and Shinnosuke {Okawa}.
\newblock {Nonexistence of semiorthogonal decompositions and sections of the
  canonical bundle}.
\newblock {\em arXiv e-prints}, page arXiv:1508.00682, Aug 2015.

\bibitem[18]{rol2015compact}
Roland Abuaf and Grégoire Menet.
\newblock Compact hyper-kähler categories, 2015.

\bibitem[19]{Kuznetsov_2016}
Alexander Kuznetsov.
\newblock Derived categories view on rationality problems.
\newblock {\em Rationality Problems in Algebraic Geometry}, page 67–104,
  2016.
\bibitem[20]{Canonaco_2017}
Alberto Canonaco and Paolo Stellari.
\newblock A tour about existence and uniqueness of dg enhancements and lifts.
\newblock {\em Journal of Geometry and Physics}, 122:28–52, Dec 2017.

\bibitem[21]{keller2006differential}
Bernhard Keller.
\newblock On differential graded categories, 2006.

\bibitem[22]{tabuada2004une}
Goncalo Tabuada.
\newblock Une structure de categorie de modeles de quillen sur la categorie des
  dg-categories, 2004.

\bibitem[23]{alex2014semiorthogonal}
\newblock {Semiorthogonal decompositions in algebraic geometry, Alexander Kuznetsov, 2014}
\newblock {arXiv math.AG 1404.3143}

\bibitem[24]{Kuznetsov_2019}
Alexander Kuznetsov.
\newblock Calabi–Yau and fractional Calabi–Yau categories.
\newblock {\em Journal für die reine und angewandte Mathematik (Crelles
  Journal) 2019(753):239–267, Aug 2019.}

\bibitem[25]{Antieau2018OnTU}
Benjamin Antieau.
\newblock On the uniqueness of infinity-categorical enhancements of
  triangulated categories.
\newblock {\em arXiv: Algebraic Geometry}, 2018.

\bibitem[26]{NM}
Goncalo Tabuada.
\newblock Noncommutative motives, University Lecture Series 63, American Mathematical Society, 2015. DOI: https://doi.org/http://dx.doi.org/10.1090/ulect/063.

\bibitem[27]{Orlov2016SMOOTHAP}
D.~Orlov.
\newblock Smooth and proper noncommutative schemes and gluing of dg categories.
\newblock {\em Advances in Mathematics}, 302:59--105, 2016.

\bibitem[28]{tabuada2019noncommutative}
Goncalo Tabuada.
\newblock Noncommutative counterparts of celebrated conjectures.
\newblock arXiv:1812.08774[math.AG].

\bibitem[29]{pirozhkov2020admissible}
Dmitrii Pirozhkov.
\newblock Admissible subcategories of del Pezzo surfaces, arXiv:2006.07643[math.AG].

\bibitem[30]{lin2021nonexistence}
Xun Lin.
\newblock On nonexistence of semi-orthogonal decompositions in algebraic geometry.
\newblock  arXiv:2107.09564[math.AG], 2021.


\end{thebibliography}
\end{document}